\documentclass[11pt]{article}

\usepackage{amsmath}
\usepackage{amssymb}
\usepackage{amsfonts}
\usepackage{amsthm}
\usepackage[a4paper]{geometry}
\usepackage{graphicx}
\usepackage[abs]{overpic}
\usepackage{caption}
\usepackage{wrapfig}
\usepackage{float}
\usepackage{graphicx}
\usepackage{bm}
\usepackage{verbatim}
\usepackage{enumitem}
\usepackage{eucal}

\usepackage{todonotes}

\usepackage[utf8]{inputenc}

\newtheorem{theo}{Theorem}[section]
\newtheorem{lemma}[theo]{Lemma}
\newtheorem{prop}[theo]{Proposition}
\newtheorem{corollary}[theo]{Corollary}

\theoremstyle{definition}
\newtheorem{rem}[theo]{Remark}
\newtheorem{definition}[theo]{Definition}

\def\E {\mathfrak{E}}
\def\En {{\cal E}}
\def\F {{\cal F}}

\def\I {{\cal I}}

\def\x {\ dx}
\def\y {\ dy}

\def\R {\mathbb R}

\def\N {\mathbb N}

\def\i {\infty}
\def\tr {\mathop{\rm tr}}
\def\* {$*$}
\def\eps {\varepsilon}
\def\id{\mathop{\rm id}}
\def\supp{\mathop{\rm supp}}

\def\qc{\mathop{\rm qc}}
\def\iqc{\mathop{\rm iqc}}
\def\sm{\mathop{\rm sym}}

\def\swlsc{\mathop{\rm swlsc}}
\def\esssup{\mathop{\rm ess\,sup}}
\def\dist{\mathop{\rm dist}}
\def\dev{\mathop{\rm dev}}
\def\ils{\mathop{\rm ils}}
\def\DIV{\mathop{\rm div}}

\def\Id{\mathbf{Id}}
\def\id{\mathbf{id}}

\def\eps{\varepsilon}

\def\SO{\mathrm{SO}}
\def\SL{\mathrm{SL}}

\def\weakly{\rightharpoonup}

\def\dist{\operatorname{dist}}

\newcommand{\M}[2]{\R^{#1 \times #2}}

\usepackage{color}

\definecolor{darkgreen}{rgb}{0, 0.5, 0} 
\definecolor{darkorange}{rgb}{1,0.5,0}

\begin{document}

%--------------------------------
%---
\begin{center}
\begin{Large}
{\bf Geometric linearization of theories for incompressible elastic materials and applications} 
%
%On the passage from nonlinear elasticity to geometrically linearized functionals for incompressible materials and an application to nematic elastomers}
\end{Large}
\end{center}

\begin{center}
\begin{large}
Martin Jesenko\footnote{Abteilung f\"{u}r Angewandte Mathematik, Mathematisches Institut, Albert-Ludwigs-Universit\"{a}t Freiburg, 79104 Freiburg, Germany,
{\tt martin.jesenko@mathematik.uni-freiburg.de}} 
and 
Bernd Schmidt\footnote{Institut f{\"u}r Mathematik, Universit{\"a}t Augsburg, 86135 Augsburg, Germany, 
{\tt bernd.schmidt@math.uni-augsburg.de}}\\[0.2cm]
\end{large}
\end{center}

\begin{center}
\today
\end{center}
\bigskip

\begin{abstract}
We derive geometrically linearized theories for incompressible materials from nonlinear elasticity theory in the small displacement regime. Our nonlinear stored energy densities may vary on the same (small) length scale as the typical displacements. This allows for applications to multiwell energies as, e.g., encountered in martensitic phases of shape memory alloys and models for nematic elastomers. Under natural assumptions on the asymptotic behavior of such densities we prove Gamma-convergence of the properly rescaled nonlinear energy functionals to the relaxation of an effective model. The resulting limiting theory is geometrically linearized in the sense that it acts on infinitesimal displacements rather than finite deformations, but will in general still have a limiting stored energy density that depends in a nonlinear way on the infinitesimal strains. Our results, in particular, establish a rigorous link of existing finite and infinitesimal theories for incompressible nematic elastomers. 
\end{abstract}
\bigskip

\begin{small}

\medskip

\end{small}

%\tableofcontents

%--------------------------------------------------------------------------
%--------------------------------------------------------------------------
\section{Introduction}\label{sec:Intro}
%--------------------------------------------------------------------------

The relation between nonlinear and linear models is a classical topic in elasticity theory. While for finite (large or moderate) deformations the stress-strain relation of an elastic material is nonlinear, in the regime of small displacements this relation is typically given by a linear diagram. Dealing with hyperelastic materials, which can be described by their associated stored energy functionals, we encounter a nonlinear (more precisely: nonquadratic) energy density at finite deformations, while in the small displacement regime the energy is a quadratic form that acts on the infinitesimal strains and whose coefficients are the elastic moduli of the material. Classically, the linear theory is obtained from the nonlinear one by Taylor approximation and the matrix of the elastic moduli is nothing but the Hessian of the nonlinear energy density at the identity matrix. 

Although standard, this relation has been given a rigorous corroboration only comparatively recently by Dal Maso, Negri and Percivale who proved that the linearized theory is the $\Gamma$-limit of the nonlinear energy functionals and, in particular, energy minimizers subject to suitable boundary values and body forces do indeed converge to minimizers of the limiting linear model, see \cite{DalMasoNegriPercivale}. 

These results have been generalized by the second author to multiwell energies in \cite{Schmidt:08}. In such models the nonlinear energy densities are allowed to vary on the same length scale as the typical (small) displacements. They describe, e.g., martensitic phases of shape memory alloys which are modeled by multiple nearby energy wells. As a consequence, passing to the regime of infinitesimal displacements will lead to still nonlinear energy functionals which, however, act on linearized deformations, i.e., on infinitesimal displacements. We thus obtain merely `geometrically' rather than fully linearized theories. Later we have revisited these results in \cite{JesenkoSchmidt:14} within a general $\Gamma$-commutability theory, which in addition allows for simultaneous homogenization and which had been motivated by \cite{Braides:86,MuellerNeukamm:11,GloriaNeukamm:11}. More recently, results have also been obtained for multiple wells which remain at a macroscopic distance from each other with the help of a singular perturbation, see \cite{Alicandro-etal:18}. 

From a more applied point of view, it has first been observed in \cite{AgostinianiDeSimone:11} that the theory developed in \cite{Schmidt:08} can be applied to relate certain linear and nonlinear theories for nematic elastomers. These materials are composite materials consisting of a nematic liquid crystal coupled to a rubbery polymer matrix. Their elastic behavior can be modeled by an energy density depending on the strain and a director field $\nu$ which prefers tensile strains in the direction of $\nu$. Minimizing over the internal variable $\nu$, one arrives at an energy density with infinitely many wells, parameterized by $\nu \in S^2$.  

The goal of the present work is to extend the aforementioned results of geometric linearization to models for incompressible materials. Such an extension is most desirable in view of applications. In particular, the (near) incompressibility of the rubber matrix causes the bulk modulus of a nematic elastomer to be orders of magnitude larger than the shear modulus, so that these materials are typically modeled as incompressible. From a mathematical perspective such an assumption imposes non-trivial constraints: In the nonlinear theory this amounts to requiring that the deformation gradients be volume preserving, i.e., have constant determinant equal to $1$. In the (geometrically) linearized theory this leads to the condition that the displacement be solenoidal (i.e., divergence-free). These conditions have been investigated previously also within relaxation results, see \cite{Braides} in the linear and \cite{ContiDolzmann:15,CicaleseFusco} in the nonlinear setting. In an abstract sense, our main task is to investigate their interrelation in the small displacement regime and to answer (positively) the question if linearization and imposing an incompressibility constraint do commute.

To this end, our main contribution is a $\Gamma$-convergence and a compactness result for a sequence of functionals $\F_{\eps}$ with nonlinear stored energy density $W_{\eps}$ for incompressible materials in the regime of displacements scaling with $\eps \to 0$. In view of our general set-up which allows for nonlinear energy densities varying on the scale $\eps$, e.g., caused by  multiple nearby wells, our primary aim is not to explicitly identify the energy density of the limiting functional. Instead, we explicitly determine a limiting functional $\F$ with density $V$ given in terms of a suitable limit of rescalings of $W_{\eps}$, acting on linearized strains, so that the $\F_{\eps}$ $\Gamma$-converge to its relaxation $\F_{ {\rm rel} }$. This will again be an integral functional whose density is given by the `quasiconvexification on incompressible fields' $V^{\iqc}$ of $V$.  

This procedure accounts for that fact that in our setting, while there might be a direct and straighforward connection of $W_{\eps}$ and $V$, it is in general not to be expected that the $\F_{\eps}$ $\Gamma$-converge to its $\F$ due to a lack of lower semicontinuity of the latter. In particular in our applications to multiwell energies, one in general expects the presence of fine phase mixtures, and there might be non-attainment of minimizers both in the nonlinear and in the geometrically linearized setting. Instead, our results show that in the small displacement regime the functionals $\F_{\eps}$ can be replaced effectively with the explicitly given functional $\F$, in the sense that the $\Gamma$-limit of the $\F_{\eps}$ is the relaxation $\F_{ {\rm rel} }$ of $\F$. In particular, even under suitable displacement boundary conditions and body forces, the minimal value of $\F_{ {\rm rel} }$ is the infimum of $\F$ and minimizers of $\F_{ {\rm rel} }$ correspond to (converging subsequences of) low energy sequences of $\F$. 

While our main interest lies in establishing a general link between nonlinear and geometrically linearized theories, it is also interesting to investigate special cases that allow for an explicit identification of $V^{\iqc}$. This is the case for a single smooth energy well with $\eps$-independent $W$. Indeed, most recently and independently this case has been considered by Mainini and Percivale in \cite{MaininiPercivale:20} for smooth domains by quite different methods.\footnote{The preprint \cite{MaininiPercivale:20} appeared on arxiv.org three days prior to our contribution. We learned about this work after finishing our results.} Here one indeed obtains that $V^{\iqc}=V$ is half the Hessian of $W$ at $\Id$, see Theorem~\ref{theo:singlewell} for a precise statement. For incompressible variants of martensite we give conditions for the validity of a linearized `well minimum formula' in Theorem~\ref{theo:multiplewell}.

Within such explicitly solvable models, we devote our special attention to establishing a relation of the theory for incompressible nematic elastomers due to Bladon, Terentjev and Warner, see \cite{BladonTerentjevWarner:93,WarnerTerentjev,DeSimoneDolzmann,ContiDolzmann:15} to linearized theories investigated in \cite{CesanaDeSimone:11,Cesana}. Due to the incompressibility constraint this is not covered by the results in \cite{AgostinianiDeSimone:11}. Besides obtaining a general link in this setting, as a byproduct of our analysis we also find some novel explicit relaxation formulae that complete the picture obtained in \cite{CesanaDeSimone:11,Cesana}.

%%%%%%%%%%%%%%%%%%%%%%%%%%%%%%%%%%%%%%%%%%%%%%%%%%%%%%%%%%%%%%%%%%%%%%%%%%%%%%%%%%%%%%%%%%%%%%%%%%%%%%%%%%%%%%%%%%%%%%%%%
%%%%%%%%%%%%%%%%%%%%%%%%%%%%%%%%%%%%%%%%%%%%%%%%%%%%%%%%%%%%%%%%%%%%%%%%%%%%%%%%%%%%%%%%%%%%%%%%%%%%%%%%%%%%%%%%%%%%%%%%%
\section{Main convergence results}
%%%%%%%%%%%%%%%%%%%%%%%%%%%%%%%%%%%%%%%%%%%%%%%%%%%%%%%%%%%%%%%%%%%%%%%%%%%%%%%%%%%%%%%%%%%%%%%%%%%%%%%%%%%%%%%%%%%%%%%%%
%%%%%%%%%%%%%%%%%%%%%%%%%%%%%%%%%%%%%%%%%%%%%%%%%%%%%%%%%%%%%%%%%%%%%%%%%%%%%%%%%%%%%%%%%%%%%%%%%%%%%%%%%%%%%%%%%%%%%%%%%

Let the reference configuration of some incompressible material be given by a bounded Lipschitz domain $ \Omega \subset \R^{n} $, and let $ 1 < p < \i $.
We suppose that for small $ \eps > 0 $ its stored energy is given by a (nonlinear) density  
$ W_{ \eps } : \Omega \times \M{n}{n} \to \R \cup \{+\i\} $ with the following properties: 
\begin{enumerate}[label=(NL\arabic*), leftmargin=15mm]
\item regularity: 
$ W_{\eps} $ is a Borel function, 
\item incompressibility: 
$ W_{\eps}(x,X) = + \i $ if $ X \not\in \SL(n) $, 
\item frame-indifference: 
$ W_{\eps}(x,RX) = W_{\eps}(x,X) $ for a.e.~$ x \in \Omega $, every $ \eps > 0 $ and all $ R \in \SO(n) $ and $ X \in \SL(n) $, 
\item growth assumption from below: 
there exist $ \alpha, \beta > 0 $ such that
\[ W_{ \eps }( x , X ) \ge \alpha \, {\dist}^{p}( X , \SO(n) ) - \beta \, \eps^{p} \]
for a.e.~$ x \in \Omega $, every $ \eps > 0 $ and all $ X \in \SL(n) $.
\end{enumerate}
Here $\SL(n)$ denotes the special linear group $\{X \in \M{n}{n} : \det X = 1\}$, and the physically most relevant case is $p = 2$. Such densities include incompressible versions of \cite{DalMasoNegriPercivale, Schmidt:08}, 
i.e.~one-well energy densities with a non-degenerate minimum at rotations and multiple-well densities with several minima lying in an $ \eps $-neighbourhood of $ \SO(n) $. 

We suppose that the material satisfies a Dirichlet boundary condition at some portion $ \partial \Omega_{*} $ of its boundary $ \partial \Omega $. 
(We allow $ \partial \Omega_{*} = \emptyset $ though.)
For a given sequence $w_{\eps} \in W^{1,\infty}( \Omega ; \R^{n} ) $ with $\det \nabla w_{\eps} = 1$ a.e.\ for each $\eps$, the set of admissible deformations is $w_{\eps} + W^{1,p}_{ \partial \Omega_{*} }( \Omega ; \R^{n} ) $ where, following \cite{BMMM}, we define for any closed subset $ D \subset \overline{ \Omega } $
the space $ W^{1,p}_{ D }( \Omega ; \R^{n} ) $ as the closure of 
$\{ u|_{ \Omega } : u \in C_{c}^{\i}( \R^{n} \setminus D ; \R^{n} ) \} $ in $ W^{1,p}( \Omega ; \R^{n} ) $.

If we apply the external load $ \ell_{\eps} $ to such a material, subject to a deformation $y \in w_{\eps} + W^{1,p}_{ \partial \Omega_{*} }( \Omega ; \R^{n} )$, then its total elastic energy consists of the deformation energy and the potential energy of applied forces, i.e.
\[ \int_{ \Omega } W_{\eps}( x , \nabla y(x) ) \x - \int_{ \Omega } \ell_{\eps}(x) \cdot y(x) \x. \]
We will assume that there is no net force or first moment by requiring 
\[ \int_{\Omega} \ell_{\eps}(x) \x = 0 
   \quad\mbox{and}\quad 
   \int_{\Omega} \ell_{\eps}(x) \otimes x \x= 0. \] 
(If $ {\cal H}^{n-1}( \partial \Omega_{*} ) > 0 $, the latter assumption can be weakened to $\int_{\Omega} \ell_{\eps}(x) \cdot x \x = 0$.) If the load and the boundary displacements are small, i.e.~ $ \ell_{\eps} = \eps^{p-1} \tilde{\ell}_{\eps} $ and $ w_{\eps} = \id + \eps g_{\eps}$ for some $ \tilde{\ell}_{\eps} \in L^{p'}( \Omega ; \R^{n} ) $ and $ g_{\eps} \in W^{1,\infty}( \Omega ; \R^{n} )$ of order $1$ as $ \eps \to 0 $, then we also expect the resulting deformations to be close to rigid motions. We accordingly rescale the energy as 
\[ \En_{ \eps }(y) 
:= \begin{cases} 
     \frac{1}{ \eps^{p} } \int_{ \Omega } W_{\eps}( x , \nabla y(x) ) \x - \frac{1}{ \eps } \int_{ \Omega } \tilde{\ell}_{\eps}(x) \cdot y(x) \x, 
    &\mbox{if } y \in w_{\eps} + W_{ \partial \Omega_{*} }^{1,p}( \Omega ; \R^{n} ),  \\ 
\i, & \mbox{otherwise}.
\end{cases} \]
Indeed as we will see, geometric rigidity implies that deformations $y$ with energy $\En_{ \eps }(y)$ of order $1$ are in fact close to a rigid motion $x \mapsto R (x + c)$, where $R \in \SO(n)$, $c \in \R^n$. In particular, if $y(x) = w_{\eps}(x)$ on a non-negligible part of the boundary, we may choose $R = \Id$. We introduce the rescaled displacement $u$ (w.r.t.\ this rigid motion) as 
\[ u(x) = \frac{R^T y(x) - x - c}{\eps} \] 
and observe that 
\[ \En_{\eps}(y) 
   = \frac{1}{ \eps^{p} } \int_{ \Omega } W_{\eps}( x , \Id + \eps \nabla u(x) ) \x 
     - \int_{ \Omega } \tilde{\ell}_{\eps}(x) \cdot R u(x) \x\] 
where we have used the frame indifference of $W_{\eps}$ and the assumptions on $\ell_{\eps}$.

Our aim is to derive an effective `geometrically linearized' model for such deformations, i.e., to pass to a (possibly nonlinear) limiting integral functional acting on linearized strains. These strains lie in the tangent space of $ \SL(n) $ at $ \Id $ (in the Lie algebra of $ \SL(n) $), which is the space of all deviatoric matrices $ \M{n}{n}_{ \dev } := \{ X \in \M{n}{n} : \tr X = 0 \} $. In order to state our assumptions on the limiting densities, we use the following notation. By $ X_{\sm} := \tfrac{1}{2}( X + X^{T} ) $ we denote the projection onto the space of symmetric matrices $\M{n}{n}_{\sm}$. The space of incompressible linear strains is $ \M{n}{n}_{\sm} \cap \M{n}{n}_{\dev} =: \M{n}{n}_{\ils} $. Let us denote the projections onto $ \M{n}{n}_{\dev} $ and $ \M{n}{n}_{\ils} $ by
\[ 
X_{\dev} := X - \tfrac{\tr X}{n} \Id 
\quad\mbox{and}\quad
X_{\ils} := X_{\sm} - \tfrac{\tr X}{n} \Id,
\quad \mbox{respectively.}
\]
For the projected gradients we write 
\[
\E u := ( \nabla u )_{\sm},
\quad
\nabla_{\dev} u := ( \nabla u )_{\dev}
\quad \mbox{and} \quad
\E_{\dev} u := ( \nabla u )_{\ils},
\quad \mbox{respectively}.  \]

Any candidate $ V : \Omega \times \M{n}{n} \to \R$ for a density of the limiting functional on linearized strains will satisfy the following conditions: 
\begin{enumerate}[label=(L\arabic*), leftmargin=15mm]
\item regularity: 
$V$ is a Borel function satsifying the local Lipschitz property 
\[ | V(x,Z) - V(x,Z') | \le \beta ( 1 + |Z|^{p-1}+ |Z'|^{p-1} ) | Z - Z' |. \]
for a.e.~$ x \in \Omega $ and all $ Z, Z' \in \M{n}{n}_{\dev} $, 
\item linearized incompressibility: 
$ V(x,Z) = + \i $ if $ \tr Z \ne 0 $,
\item linearized frame-indifference: 
$ V(x,Z) = V(x, Z_{\sm} ) $ for a.e.~$ x \in \Omega $ and all $ Z \in \M{n}{n} $. 
\item growth assumption from below and above: 
there exist $ \alpha, \beta > 0 $ such that
\[ \alpha \, | Z_{\ils} |^p - \beta 
   \le V( x , Z ) 
   \le \beta \, | Z |^p + \beta \]
for a.e.~$ x \in \Omega $ and all $ Z \in \M{n}{n}_{\dev} $. 
\end{enumerate}
In particular, the function $V$ is fully determined by the values on $\M{n}{n}_{\ils}$. 

For a simple, $\eps$-independent energy well at $\SO(n)$, the function $V$ arises naturally as the first non-trivial (quadratic) term in the Taylor expansion of $W$. In our more general setting we define the rescaled densities 
\[ V_{ \eps } : \Omega \times \M{n}{n}_{ \dev } \to \R, 
\quad V_{ \eps }( x , Z ) := \frac{1}{ \eps^{p} } W_{ \eps }( x , e^{ \eps Z } ) \]
in analogy to \cite{Schmidt:08}, where here in addition we also take advantage of the exponential function mapping $\M{n}{n}_{\dev}$ to $\SL(n)$ in order for the domains of $ V_{\eps} $, $ \eps > 0 $, and $V$ to coincide. Then, as it will be precisely stated in Theorem~\ref{theo:gamma}, we will require that $ V_{\eps} $ approximates $ V $ in an $L^1(\Omega; L^{\infty}_{\rm loc}(\M{n}{n}_{\dev}))$ sense. 

In case $\partial \Omega_{*} \ne \emptyset$, we will throughout assume that the boundary displacements $g_{\eps} \in W^{1,\infty}(\Omega, \R^n)$ given by $w_{\eps} = \id + \eps g_{\eps}$ satisfy 
\[ g_{\eps} \to g \qquad \mbox{in } W^{1,\infty}(\Omega, \R^n) 
\]
for some $g \in W^{1,\infty}(\Omega, \R^n)$. Note that $\det \nabla w_{\eps} = 1$ a.e.\ implies $\tr \nabla g = 0$ a.e.

Our main $\Gamma$-convergence result connecting nonlinear and geometrically linear theories is the following.  
%------------------------------------------------------------------------------------------------------------------------
\begin{theo}
\label{theo:gamma}
Suppose $W_{\eps}$ and $V$ satisfy the assumptions 
(NL1)--(NL4), resp., (L1)--(L4). Assume that 
\begin{equation}
\label{eq:conv} \lim_{\eps\to 0} \int_{\Omega} \sup_{Z \in \M{n}{n}_{\dev}, |Z| \le r} |V_{\eps}(x,Z) - V(x,Z)| \x = 0
\qquad 
\mbox{for every $r > 0$. }
\tag{C}
\end{equation}
If $ \partial \Omega_{*} \subset \partial \Omega $ is a closed subset such that $ \R^{n} \setminus \partial \Omega_{*} $ satisfies the cone property, then it holds 
\[ \Gamma \text{-} \lim_{ \eps \to 0 } \F_{ \eps } = \F_{ {\rm rel} } \]
 on $ L^{p}( \Omega ; \R^{n} ) $, where for $ \eps > 0 $ 
\[ \F_{ \eps }(u) 
:= \begin{cases}
   \frac{1}{ \eps^{p} } \int_{ \Omega } W_{ \eps }( x , \Id + \eps \nabla u(x) ) \x, &\mbox{if } u \in g_{\eps} + W_{ \partial \Omega_{*} }^{1,p}( \Omega ; \R^{n} ), \
				\det( \Id + \eps \nabla u ) = 1,  \\
\i, & \mbox{otherwise},
\end{cases} \]
and
\[ \F_{ {\rm rel} }(u) 
:= \begin{cases}
\int_{ \Omega } \overline{V}( x , \E u(x) ) \x, &\mbox{if } u \in g + W_{ \partial \Omega_{*} }^{1,p}( \Omega ; \R^{n} ), \ \DIV u = 0, \\
\i, & \mbox{otherwise}.
\end{cases} \]
The function $ \overline{V} $ is the $\iqc$-envelope of $ V|_{ \M{n}{n}_{\dev} } $, see Definition \ref{def:iqc}(b).
\end{theo}
%------------------------------------------------------------------------------------------------------------------------
\begin{rem}
\begin{enumerate}
\item Theorem~\ref{theo:iqc} below shows that indeed $\F_{ {\rm rel} }$ is the relaxation of the functional $\F$ on $ L^{p}( \Omega ; \R^{n} ) $, defined by 
\[ \F(u) 
:= \begin{cases}
\int_{ \Omega } V( x , \E u(x) ) \x, &\mbox{if } u \in g + W_{ \partial \Omega_{*} }^{1,p}( \Omega ; \R^{n} ), \ \DIV u = 0, \\
\i, & \mbox{otherwise}.
\end{cases} \]

\item Recall that a set $A \subset \R^n$ satisfies the cone property, if there is a finite cone $\Gamma = B_\rho(0) \cap \R_+ B_\rho( e_{1})$ with $0 < \rho < 1$ such that for any $x \in A$ there is a rotation $Q \in \SO(n)$ such that $x + Q \Gamma \subset A$, compare, e.g., Remark III.3.4 in \cite{Galdi}. 

\item For the sake of simplicity, we did not include the load term. If $\tilde{\ell}_{\eps} \to \ell$ in $L^{p'}(\Omega; \R^n)$, one immediately also gets $ \Gamma $-convergence of $\F_{ \eps }(u) + \int_{\Omega} \tilde{\ell}_{\eps} \cdot u \x$ to $\F_{\rm rel }(u) + \int_{\Omega} \ell \cdot u \x$. 
\end{enumerate}
\end{rem}
%------------------------------------------------------------------------------------------------------------------------

Theorem \ref{theo:gamma} is complemented by the following compactness results. If $\partial \Omega_{*} = \emptyset$, due to frame indifference we will only obtain compactness modulo rigid motions. If, on the contrary, $\partial \Omega_{*}$ is non-negligible, we obtain compactness for the rescaled displacements with $R = \Id$ and $c = 0$. 
%------------------------------------------------------------------------------------------------------------------------
\begin{theo}\label{theo:compactness}
Suppose $\tilde{\ell}_{\eps} \to \ell$ in $L^{p'}(\Omega; \R^n)$ and $(y_{\eps}) \subset W^{1,p}( \Omega ; \R^{n} )$ satisfies $\En_{\eps}(y_{\eps}) \le C$ for a constant $C > 0$. 
\begin{enumerate}
\item There exist rotations $R_{\eps} \in \SO(n)$ and vectors $c_{\eps} \in \R^n$ such that $u_{\eps}(x) := \frac{1}{\eps}(R_{\eps}^T y_{\eps}(x) - x - c_{\eps})$ admits a subsequence which converges weakly in $W^{1,p}( \Omega ; \R^{n} )$ and thus strongly in $L^{p}( \Omega ; \R^{n} )$. 

\item If $ {\cal H}^{n-1}( \partial \Omega_{*} ) > 0 $, then $u_{\eps}(x) := \frac{1}{\eps}(y_{\eps}(x) - x)$ admits a subsequence which converges weakly in $W^{1,p}( \Omega ; \R^{n} )$ and thus strongly in $L^{p}( \Omega ; \R^{n} )$ to some $u \in g + W_{\partial \Omega_{*}}^{1,p}( \Omega ; \R^{n} )$. 
\end{enumerate}
\end{theo} 
%------------------------------------------------------------------------------------------------------------------------
As a consequence we have the following convergence properties of low energy sequences. 
\begin{corollary}\label{cor:low-energy-seq}
Suppose $\tilde{\ell}_{\eps} \to \ell$ in $L^{p'}(\Omega; \R^n)$ and $(y_{\eps}) \subset W^{1,p}( \Omega ; \R^{n} )$ satisfies 
\[ \lim_{\eps \to 0} \big( \En_{\eps}(y_{\eps}) - \inf_{w \in L^p(\Omega;\R^n)} \En_{\eps}(w) \big) = 0. \]  
\begin{enumerate}
\item If $ \partial \Omega_{*} = \emptyset $, then there exist rotations $R_{\eps} \in \SO(n)$ and vectors $c_{\eps} \in \R^n$ such that, for a subsequence, $u_{\eps} := \frac{1}{\eps}(R_{\eps}^T y_{\eps} - \id - c_{\eps}) \weakly u_0$ in $W^{1,p}( \Omega ; \R^{n} )$, $R_{\eps} \to R_0 \in \SO(n)$ and the pair $(u_0,R_0)$ is a minimizer of the functional $ {\cal G}^{(\ell)}_{ {\rm rel} } $ definied on $ L^{p}( \Omega ; \R^{n} ) \times \SO(n) $ by 
\[ {\cal G}^{(\ell)}_{ {\rm rel} } (u, R) 
:= \begin{cases}
\int_{ \Omega } ( \overline{V}( x , \E u(x) ) - \ell(x) \cdot R u(x) ) \x, & \mbox{if } u \in W^{1,p}( \Omega ; \R^{n} ), \ \DIV u = 0,   \\
\i, & \mbox{otherwise}.
\end{cases} \]
Moreover, $\lim_{\eps \to 0} \En_{\eps}(y_{\eps}) = \min {\cal G}^{(\ell)}_{ {\rm rel} }$. 
\item If $ {\cal H}^{n-1}( \partial \Omega_{*} ) > 0 $, then, for a subsequence, $u_{\eps} := \frac{1}{\eps}(y_{\eps} - \id) \weakly u_0$ in $W^{1,p}( \Omega ; \R^{n} )$ with $u_0 \in g + W_{\partial \Omega_{*}}^{1,p}( \Omega ; \R^{n} )$ and $u_0$ is a minimizer of 
\[ \F^{(\ell)}_{ {\rm rel} } (u) 
:= \begin{cases}
\int_{ \Omega } ( \overline{V}( x , \E u(x) ) - \ell(x) \cdot u(x) ) \x, &\mbox{if } u \in g + W_{ \partial \Omega_{*} }^{1,p}( \Omega ; \R^{n} ), \ \DIV u = 0, \\
\i, & \mbox{otherwise}.
\end{cases} \]
Moreover, $\lim_{\eps \to 0} \En_{\eps}(y_{\eps}) = \min \F^{(\ell)}_{ {\rm rel} }$. 
\end{enumerate}
\end{corollary}

%%%%%%%%%%%%%%%%%%%%%%%%%%%%%%%%%%%%%%%%%%%%%%%%%%%%%%%%%%%%%%%%%%%%%%%%%%%%%%%%%%%%%%%%%%%%%%%%%%%%%%%%%%%%%%%%%%%%%%%%%
%%%%%%%%%%%%%%%%%%%%%%%%%%%%%%%%%%%%%%%%%%%%%%%%%%%%%%%%%%%%%%%%%%%%%%%%%%%%%%%%%%%%%%%%%%%%%%%%%%%%%%%%%%%%%%%%%%%%%%%%%
\section{Quasiconvexity on incompressible fields}
%%%%%%%%%%%%%%%%%%%%%%%%%%%%%%%%%%%%%%%%%%%%%%%%%%%%%%%%%%%%%%%%%%%%%%%%%%%%%%%%%%%%%%%%%%%%%%%%%%%%%%%%%%%%%%%%%%%%%%%%%
%%%%%%%%%%%%%%%%%%%%%%%%%%%%%%%%%%%%%%%%%%%%%%%%%%%%%%%%%%%%%%%%%%%%%%%%%%%%%%%%%%%%%%%%%%%%%%%%%%%%%%%%%%%%%%%%%%%%%%%%%

Let us first investigate the limiting functionals. 
Since $ \Gamma $-limits are always lower semicontinuous,
we begin by exploring the relaxation of functionals of the form
\[ \F(u) = 
\left\{
\begin{array}{cl}
\int_{ \Omega } f( x , \nabla u(x) ) \x, & \mbox{if } u \in W^{1,p}( \Omega ; \R^{n} ) \mbox{ and } \DIV u = 0, \\
\i, & \mbox{else},
\end{array} \right. \]
on $ L^{p}( \Omega ; \R^{n} ) $ for some density $ f : \Omega \times \M{n}{n}_{\dev} \to \R $.
%------------------------------------------------------------------------------------------------------------------------

The relaxation of the functional without the constraint on the divergence is given by the quasiconvex envelope of $f$. 
Let us therefore introduce this concept for our setting and present analogies. 
%------------------------------------------------------------------------------------------------------------------------
\begin{definition}\label{def:iqc}
\begin{enumerate}
\item
A Borel function $ f : \M{n}{n}_{\dev} \to \R \cup \{ \i \} $ 
is {\em quasiconvex on incompressible fields}, abbr.~{\em iqc,} if for every $ X \in \M{n}{n}_{\dev} $ it holds
\[ \int_{ A } f( X + \nabla \varphi(y) ) \y \ge |A| f(X) \] 
for every bounded Lipschitz domain $ A \subset \R^{n} $ and every $ \varphi \in C_{c}^{\i}( A ; \R^{n} ) $ with $ \DIV \varphi = 0 $.
\item
For any Borel function $ f $ on $ \M{n}{n}_{\dev} $, we define its {\em iqc-envelope} by 
\[ f^{\iqc}( X ) := \sup \{ g(X) : g \le f \mbox{ is iqc}\}. \] 
\item
For a function $f$ defined on $ \Omega \times \M{n}{n}_{\dev} $, 
both definitions above are to be understood for $ f(x, \cdot ) $ for a.e.~$ x \in \Omega $.
\end{enumerate}
\end{definition}

(As in the unconstrained case, it would be sufficient to consider a single domain in (a) and more general domains can be considered.)

%------------------------------------------------------------------------------------------------------------------------
\begin{theo}
\label{theo:iqc}
Let $ \Omega \subset \R^{n} $ be a bounded Lipschitz domain
and $ \partial \Omega_{*} \subset \partial \Omega $ a closed (possibly empty) subset.
Suppose $ f: \Omega \times \M{n}{n}_{\dev} \to \R $ is a Carath\'{e}odory function 
such that for some $ p > 1 $ and $ \alpha, \beta > 0 $
\[ \alpha | X_{\ils} |^{p} - \beta \le f(x,X) \le \beta( |X|^{p} + 1 ) \]
and 
\[ | f(x,X) - f(x,Y) | \le \beta ( 1 + |X|^{p-1} + |Y|^{p-1} ) | X - Y | \]
for almost all $ x \in \Omega $ and all $ X , Y \in \M{n}{n}_{\dev} $.
Moreover, let $ u_{0} \in W^{1,p}( \Omega ; \R^{n} ) $ fulfil $ \DIV u_{0} = 0 $.
The lower semicontinuous envelope of the functional $ \I $ on $ L^{p}( \Omega ; \R^{n} ) $, given by
\[ \I(u) := 
\left\{
\begin{array}{cl}
\int_{ \Omega } f( x , \nabla u(x) ) \x, & \mbox{if } u \in u_{0} + W^{1,p}_{\partial \Omega_{*}}( \Omega ; \R^{n} ) \mbox{ with } \DIV u = 0, \\
\i, & \mbox{else},
\end{array} \right. \]
is the functional
\[ \overline{\I}(u) := 
\left\{
\begin{array}{cl}
\int_{ \Omega } f^{\iqc}( x , \nabla u(x) ) \x, & \mbox{if } u \in u_{0} + W^{1,p}_{\partial \Omega_{*}}( \Omega ; \R^{n} ) \mbox{ with } \DIV u = 0, \\
\i, & \mbox{else.}
\end{array} \right. \]
Moreover, the iqc-envelope of $f$ is for a.e.~$ x \in \Omega $ and every $ X \in \M{n}{n}_{\dev} $ given by
\[ f^{\iqc}(x,X) = \inf_{ \varphi \in C_{c}^{\i}( A ; \R^{n} ) \atop \DIV \varphi = 0 }
		\frac{1}{|A|} \int_{A} f( x , X + \nabla \varphi(y) ) \y \] 
where $ A \subset \R^{n} $ is an arbitrary bounded Lipschitz domain.
\end{theo}
%------------------------------------------------------------------------------------------------------------------------
Note that if $f(x, X) = f(x,X_{\sm})$ for a.e.~$ x \in \Omega $ and every $ X \in \M{n}{n}_{\dev} $, then also $f^{\iqc}(x, X)= f^{\iqc}(x,X_{\sm})$.
%------------------------------------------------------------------------------------------------------------------------
\begin{rem}
The corresponding results without the incompressibility constraint are standard and can be found, e.g., in \cite{Dacorogna}. Indeed, also the case of divergence-free fields has been already addressed in \cite{Braides, BraidesFonsecaLeoni,Cesana}. In \cite{Braides, BraidesFonsecaLeoni}, the authors consider the relaxation in the weak topology of $ W^{1,p} $ and show the results even without the Lipschitz assumption. For the convenience of the reader and since the bounds and the domains in our setting do not coincide with theirs, we include a self-contained proof, relying on the techniques from \cite{Braides}. Furthermore, we will later need the following two lemmas, upon which the proof rests.
\end{rem}

%------------------------------------------------------------------------------------------------------------------------

In Lemmas~\ref{lemma:iqc:1}~and~\ref{lemma:iqc:2}, we suppose that $f$ and $ \Omega $ fulfil the assumptions of Theorem~\ref{theo:iqc}.
Fix any positive sequence $ b_{j} \nearrow \i $. We define for every $ j \in \N $
\[ f_{j} : \Omega \times \M{n}{n} \to \R, \quad 
f_{j}(x,X) := f(x,X_{\dev}) + b_{j} | \tr X |^{p}  \]
and $ \overline{f} := \sup_{j} f_{j}^{\qc} $.

%------------------------------------------------------------------------------------------------------------------------

\begin{lemma}
\label{lemma:iqc:1}
For every solenoidal $ u \in W^{1,p}( \Omega ; \R^{n} ) $, 
there exist solenoidal fields $ u_{j} \in u + C_{c}^{\i}( \Omega ; \R^{n} ) $ such that
\[ \lim_{j \to \i} \| u_{j} - u \|_{ L^{p} } = 0 
\quad \mbox{and} \quad 
\lim_{j \to \i} \int_{\Omega} f( x , \nabla u_{j}(x) ) \x 
= \int_{\Omega} \overline{f}( x , \nabla u(x) ) \x. \]
\end{lemma}
\begin{proof}
Let us take any $ u \in W^{1,p}( \Omega ; \R^{n} ) $ with $ \DIV u = 0 $. 
Each $ f_{j} $ satisfies a $p$-growth condition from above. 
Therefore, we may apply Lemma~\ref{lemma:Dacorogna}. 
There exist fields $ z_{j} \in W^{1,p}_{0}( \Omega ; \R^{n} ) $ such that
\[ \| z_{j} \|_{ L^{p} } \le \frac{1}{j}
\quad \mbox{and} \quad
\int_{\Omega} f_{j}( x , \nabla U_{j} (x) ) \x 
\le \int_{\Omega} f_{j}^{\qc}( x , \nabla u(x) ) \x + \frac{1}{j} \]
with $ U_{j} := u + z_{j} $. By the continuity and growth properties of $f_{j}$, we may even assume that  $ z_{j} \in C^{\infty}_{c}( \Omega ; \R^{n} ) $. 
The functions $ U_{j} $, however, may not be solenoidal. 
For all $ j \in \N $, we have on the one hand
\begin{align*}
  \int_{\Omega} f_{j}( x , \nabla U_{j}(x) ) \x 
  &\le \int_{\Omega} f_{j}^{\qc}( x , \nabla u(x) ) \x + 1 \\ 
  &\le \int_{\Omega} f_{j}( x , \nabla u(x) ) \x + 1
  = \int_{\Omega} f ( x , \nabla u(x) ) \x + 1, 
\end{align*}
and on the other hand
\begin{align*}
  \int_{\Omega} f_{j}( x , \nabla U_{j}(x) ) \x  
  &\ge \int_{\Omega} \Big( \alpha | \E_{ \dev } U_{j}(x) |^{p} - \beta + b_{j} | \DIV U_{j} |^{p} \Big) \x \\ 
  &\ge \int_{\Omega} \Big( c | \E U_{j}(x) |^{p} - \beta \Big) \x. 
\end{align*}
By Korn's inequality, $ ( U_{j} )_{ j \in \N } $ and therefore also $ ( z_{j} )_{ j \in \N } $ 
are bounded in $ W^{1,p}( \Omega ; \R^{n} ) $.
Moreover, $ b_{j} \| \DIV U_{j} \|_{ L^{p} } = b_{j} \| \DIV z_{j} \|_{ L^{p} } $ is bounded uniformly in $j$. 
With the help of the Bogovskii operator 
$\mathcal{B} : \{ z \in L^p(\Omega) : \int_{\Omega} z(x) \x = 0\} \to W^{1,p}_0(\Omega; \R^n)$ 
(see Theorem~\ref{theo:Bogovskii}), we define $w_j := z_j - \mathcal{B} (\DIV z_j) \in C^{\infty}_c(\Omega; \R^n)$ 
and obtain the solenoidal functions $u_j \in u + C^{\infty}_c(\Omega; \R^n)$ by setting $u_j := u + w_j$. 
By Theorem~\ref{theo:Bogovskii} and Poincare's inequality, we have
\[ \| u_j - U_j \|_{W^{1,p}}
   = \| \mathcal{B} (\DIV z_j) \|_{W^{1,p}}
   \le C \| \DIV z_j \|_{L^p}
   \le C b_{j}^{-1}. \]
From
\[ \big| f( x, \nabla_{\dev} U_{j}(x) ) - f( x , \nabla u_{j}(x) ) \big| 
   \le \beta \big(  1 + | \nabla_{\dev} U_{j}(x) |^{p-1} + | \nabla u_{j}(x) |^{p-1} \big) 
	  \big| \nabla_{\dev}( U_{j} - u_{j} )(x) \big| \]
and H{\"o}lder's inequality, it follows that
\begin{align*}
  &\int_{ \Omega } \big| f( x , \nabla_{\dev} U_{j}(x) ) - f( x , \nabla u_{j}(x) ) \big| \x \\ 
  &~~ \le C \big( 1 + \| \nabla_{\dev} U_{j} \|_{L^{p}}^{p-1} + \| \nabla u_{j} \|_{L^{p}}^{p-1} \big) 
        \| \nabla_{\dev} (U_{j} - u_{j}) \|_{L^{p}}
   \to 0 
\end{align*} 
as $j \to \infty$. Therefore,
\begin{align*}
\limsup_{ j \to \i } \int_{\Omega} f( x , \nabla u_{j}(x) ) \x 
&  =  \limsup_{ j \to \i } \int_{\Omega} f( x , \nabla_{\dev} U_{j}(x) ) \x \\
& \le \limsup_{ j \to \i } \int_{\Omega} f_{j}( x , \nabla U_{j}(x) ) \x \\
& \le \limsup_{ j \to \i } \int_{\Omega} f_{j}^{\qc}( x , \nabla u(x) ) \x \\
& = \int_{\Omega} \overline{f}( x , \nabla u(x) ) \x 
\end{align*}
by the monotone convergence theorem.
On the other hand, we note that $U_j \weakly u$ in $W^{1,p}( \Omega ; \R^{n} )$ 
since $U_j \to u$ in $L^p( \Omega ; \R^{n} )$ and $ ( U_{j} )_{ j \in \N } $ is 
bounded in $W^{1,p}( \Omega ; \R^{n} )$, 
whence also $u_j \weakly u$ in $W^{1,p}( \Omega ; \R^{n} )$ and $u_j \to u$ in $L^p( \Omega ; \R^{n} )$. 
So by Theorem~\ref{theo:lsc} and Korn's inequality, for every $ k \in \N $ it holds 
\begin{align*}
\liminf_{ j \to \i } \int_{\Omega} f( x , \nabla u_{j}(x) ) \x 
&  =  \liminf_{ j \to \i } \int_{\Omega} f_{k}( x , \nabla u_{j}(x) ) \x \\
& \ge \liminf_{ j \to \i } \int_{\Omega} f_{k}^{\qc}( x , \nabla u_{j}(x) ) \x \\
& \ge \int_{\Omega} f_{k}^{\qc}( x , \nabla u(x) ) \x. 
\end{align*}
Applying the monotone convergence theorem again, we see that the functions $ u_{j} \in u + C^{\infty}_c(\Omega; \R^n)$ with $\DIV u_j = 0$ and $u_j \to u$ in $L^p( \Omega ; \R^{n} )$ satisfy  
\[ \lim_{ j \to \i } \int_{\Omega} f( x , \nabla u_{j}(x) ) \x
  = \sup_{ k \in \N } \int_{\Omega} f_{k}^{\qc}( x , \nabla u(x) ) \x 
  = \int_{\Omega} \overline{f}( x , \nabla u(x) ) \x.  \qedhere
\]
\end{proof}

\begin{lemma}
\label{lemma:iqc:2}
For a.e.~$ x_{0} \in \Omega $ it holds:
\begin{enumerate}
\item 
If $ X_{0} \in \M{n}{n}_{\dev} $, then
\[ f^{\iqc}( x_{0}, X_{0} ) = \overline{f}( x_{0}, X_{0} ) = \inf_{ \varphi \in C_{c}^{\i}( A ; \R^{n} ) \atop \DIV \varphi = 0 }
		\frac{1}{|A|} \int_{ A } f( x_{0}, X_{0} + \nabla \varphi(y) ) \y, \]
\item
whereas for $ X_{0} \not\in \M{n}{n}_{\dev} $, we have $ \overline{f}( x_{0}, X_{0} ) = \i $.
\end{enumerate} 

\end{lemma}

\begin{proof}
Let us fix any $ x_{0} \in \Omega $ such that $ f( x_{0} ,\cdot) $ satisfies the local Lipschitz and the growth condition stated in Theorem~\ref{theo:iqc} and any bounded Lipschitz domain $ A \subset \R^{n} $. The function 
\[ A \times  \M{n}{n}_{\dev} \to \R, \quad
(x, X) \mapsto f(x_{0}, X) =: g(X) \]
(independent of $x$) and the set $A$ fulfil the conditions of Theorem~\ref{theo:iqc}. 
Let us choose any $ X_{0} \in \M{n}{n}_{\dev} $. 
The affine function $ x \mapsto X_{0} x $ is solenoidal. 
Therefore, we may apply Lemma~\ref{lemma:iqc:1} to obtain solenoidal $ \varphi_{j} \in C_{c}^{\i}( A ; \R^{n} ) $ for which 
\begin{align*}
\lim_{ j \to \i } \int_{ A } g ( X_{0} + \nabla \varphi_{j}(y) ) \y
  = \int_{A} \overline{g}(X_{0}) \y
  = |A| \overline{g}(X_{0}). 
\end{align*}
From the definitions, it follows $ \overline{g} = \overline{f}( x_{0} , \cdot ) $. Hence, 
\begin{align*}
\overline{f} ( x_{0} , X_{0} )  
&  = \lim_{ j \to \i } \frac{1}{|A|} \int_{ A } f( x_{0} , X_{0} + \nabla \varphi_{j}(y) ) \y \\
& \ge \limsup_{ j \to \i } \frac{1}{|A|} \int_{A} f^{\iqc}( x_{0} , X_{0} + \nabla \varphi_{j}(y) ) \y 
 \ge f^{\iqc} ( x_{0} , X_{0} ). 
\end{align*}
For each $ j \in \N $ clearly $ f_{j}^{\qc}|_{ \Omega \times \M{n}{n}_{\dev} } \le f_{j}|_{ \Omega \times \M{n}{n}_{\dev} } = f $. Since $ f_{j}^{\qc}|_{ \Omega \times \M{n}{n}_{\dev} } $ is $\iqc$, by definition we have 
\[ f_{j}^{\qc}|_{ \Omega \times \M{n}{n}_{\dev} } \le f^{\iqc}. \]
Hence, $ \overline{f} ( x_{0} , X_{0} ) = f^{\iqc} ( x_{0} , X_{0} ) $.
The formula holds since
\[ \overline{f} ( x_{0} , X_{0} )
= \lim_{ j \to \i } \frac{1}{|A|} \int_{A} f( x_{0} , X_{0} + \nabla \varphi_{j}(y) ) \y
\ge \inf_{ \varphi \in C_{c}^{\i}( A ; \R^{n} ) \atop \DIV \varphi = 0 }
		\frac{1}{|A|} \int_{A} f( x_{0} , X_{0} + \nabla \varphi(y) ) \y
\] 
while for every $ j \in \N $ we have also 
\begin{align*}
f_{j}^{\qc}(x_{0},X_{0})
&  = \inf_{ \varphi \in C_{c}^{\i}( A ; \R^{n} ) } \frac{1}{|A|} \int_{A} f_{j}( x_{0} , X_{0} + \nabla \varphi(y) ) \y \\
& \le \inf_{ \varphi \in C_{c}^{\i}( A ; \R^{n} ) \atop \DIV \varphi = 0 }
		\frac{1}{|A|} \int_{A} f_{j}( x_{0} , X_{0} + \nabla \varphi(y) ) \y \\
&  =  \inf_{ \varphi \in C_{c}^{\i}( A ; \R^{n} ) \atop \DIV \varphi = 0 }
		\frac{1}{|A|} \int_{A} f( x_{0} , X_{0} + \nabla \varphi(y) ) \y.
\end{align*} 

For part (b), it remains to note that for $ X_{0} \not\in \M{n}{n}_{\dev} $, we have 
\[ f_{j}^{\qc}(x_{0},X_{0}) \ge - \beta + b_{j} | \tr X_{0}|^p \nearrow \infty. \qedhere \] 
\end{proof}

\begin{proof}[Proof of Theorem~\ref{theo:iqc}]

Let us first show the $\liminf$-inequality. 
Take arbitrary $ u \in L^{p}( \Omega ; \R^{n} ) $ and a sequence $ u_{j} \to u $. 
As usual, we may suppose that $ ( \I( u_{j} ) ) $ is bounded.
Hence, all $ u_{j} $ are solenoidal and lie in $ u_{0} + W^{1,p}_{\partial \Omega_{*}}( \Omega ; \R^{n} ) $. From 
\[ C \ge \I( u_{j} ) = \int_{\Omega} f( x , \nabla u_{j}(x) ) \x
\ge \int_{\Omega} ( \alpha | \E u_{j}(x) |^{p} - \beta ) \x \]
and Korn's inequality, it follows that $ u_{j} \weakly u $ in $ W^{1,p}( \Omega ; \R^{n} )  $. Hence, also $ u \in u_{0} + W^{1,p}_{\partial \Omega_{*}}( \Omega ; \R^{n} ) $.
Moreover, for every $ k \in \N $
\begin{align*}
\liminf_{ j \to \i } \I( u_{j} )
&  =  \liminf_{ j \to \i } \int_{\Omega} f( x , \nabla u_{j}(x) ) \x \\
&  =  \liminf_{ j \to \i } \int_{\Omega} f_{k}( x , \nabla u_{j}(x) ) \x \\
& \ge \liminf_{ j \to \i } \int_{\Omega} f_{k}^{\qc}( x , \nabla u_{j}(x) ) \x \\
& \ge \int_{\Omega} f_{k}^{\qc}( x , \nabla u(x) ) \x 
\end{align*}
by weak lower semicontinuity, and therefore, by the monotone convergence theorem, 
\[ \liminf_{ j \to \i } \I( u_{j} ) 
\ge \sup_{ k \in \N } \int_{\Omega} f_{k}^{\qc}( x , \nabla u(x) ) \x
= \int_{\Omega} f^{\iqc}( x , \nabla u(x) ) \x
= \overline{\I}(u). \]

The existence of a recovery sequence for $ u \in u_{0} + W^{1,p}_{\partial \Omega_{*}}( \Omega ; \R^{n} ) $ follows from Lemmas~\ref{lemma:iqc:1}~and~\ref{lemma:iqc:2}. For all other functions, this is a consequence of the $ \liminf $-inequality.
\end{proof}

%%%%%%%%%%%%%%%%%%%%%%%%%%%%%%%%%%%%%%%%%%%%%%%%%%%%%%%%%%%%%%%%%%%%%%%%%%%%%%%%%%%%%%%%%%%%%%%%%%%%%%%%%%%%%%%%%%%%%%%%%
%%%%%%%%%%%%%%%%%%%%%%%%%%%%%%%%%%%%%%%%%%%%%%%%%%%%%%%%%%%%%%%%%%%%%%%%%%%%%%%%%%%%%%%%%%%%%%%%%%%%%%%%%%%%%%%%%%%%%%%%%
\section{Proof of the main convergence results}
%%%%%%%%%%%%%%%%%%%%%%%%%%%%%%%%%%%%%%%%%%%%%%%%%%%%%%%%%%%%%%%%%%%%%%%%%%%%%%%%%%%%%%%%%%%%%%%%%%%%%%%%%%%%%%%%%%%%%%%%%
%%%%%%%%%%%%%%%%%%%%%%%%%%%%%%%%%%%%%%%%%%%%%%%%%%%%%%%%%%%%%%%%%%%%%%%%%%%%%%%%%%%%%%%%%%%%%%%%%%%%%%%%%%%%%%%%%%%%%%%%%
%------------------------------------------------------------------------------------------------------------------------

Now we prove our main convergence and compactness results Theorem~\ref{theo:gamma}, Theorem~\ref{theo:compactness} and Corollary~\ref{cor:low-energy-seq}.

%%%%%%%%%%%%%%%%%%%%%%%%%%%%%%%%%%%%%%%%%%%%%%%%%%%%%%%%%%%%%%%%%%%%%%%%%%%%%%%%%%%%%%%%%%%%%%%%%%%%%%%%%%%%%%%%%%%%%%%%%
\subsection{The liminf-inequality}
%%%%%%%%%%%%%%%%%%%%%%%%%%%%%%%%%%%%%%%%%%%%%%%%%%%%%%%%%%%%%%%%%%%%%%%%%%%%%%%%%%%%%%%%%%%%%%%%%%%%%%%%%%%%%%%%%%%%%%%%%

%
\begin{proof}[Proof of the $\liminf$-inequality in Theorem~\ref{theo:gamma}]
For any $ \kappa > 0 $ we define a frame indifferent approximation $W_{\eps,\kappa}$ of $W_{\eps}$ on $ \Omega \times \M{n}{n}$ with good growth behavior at $\infty$ by setting 
\[ \widetilde{W}_{\eps,\kappa}(x,X) 
   := \begin{cases} 
	      W_{\eps} \big( x , (\det X)^{-1/n} X \big) + \kappa \big| (\det X)^{-1/n} X - X \big|^p, &\mbox{if } \det X > 0, \\
				+\infty, &\mbox{if } \det X \le 0, 
		 \end{cases} \] 
and 
\[ W_{\eps,\kappa}(x,X) 
   := \min \big\{ \widetilde{W}_{\eps,\kappa}(x,X), \kappa^2 \dist^p ( X, \SO(n) ) + \kappa^2 \eps^p \big\}. \] 
The corresponding symmetric approximation of $V$ is
\[ V_{\kappa}(x,X) 
   := V( x , X_{\dev} ) + \kappa n^{-p/2}| \tr X |^p. \]
Notice that for every $X \in \M{n}{n}$ with $\det X > 0$ 
\begin{align*}
  \dist^p ( X, \SO(n) ) 
  \le 2^{p-1} \dist^p \big( (\det X)^{-1/n} X, \SO(n) \big) + 2^{p-1} \big| (\det X)^{-1/n} X - X \big|^p,
\end{align*}
and so for $ c = 2^{1-p} \min \{ \alpha , \kappa \} > 0 $ 
\begin{align*}
  \widetilde{W}_{\eps,\kappa}(x,X) 
  &\ge \alpha \dist^p \big( (\det X)^{-1/n} X, \SO(n) \big) - \beta \eps^p + \kappa \big| (\det X)^{-1/n} X - X \big|^p \\ 
  &\ge c \dist^p ( X, \SO(n) ) - \beta \eps^p,  
\end{align*}
whence for a.e.\ $x \in \Omega$ and all $X \in \M{n}{n}$ also 
\begin{align*}
  W_{\eps,\kappa}(x,X) 
  \ge c \dist^p ( X, \SO(n) ) - \beta \eps^p. 
\end{align*}
For $ x \in \Omega $ and $ Y \in \M{n}{n} $, we set 
\[ \widetilde{V}_{\eps,\kappa} (x,Y) 
   := \frac{1}{ \eps^{p} } \widetilde{W}_{\eps,\kappa} ( x , \Id + \eps Y ) 
   \quad\mbox{and}\quad 
   V_{\eps,\kappa} (x,Y) 
   := \frac{1}{ \eps^{p} } W_{\eps,\kappa} ( x , \Id + \eps Y ).
\] 
Let $r > 0$ and $Y \in \M{n}{n}_{\sm}$ with $|Y| \le r$. Since $\M{n}{n}_{\rm dev}$ is the Lie algebra of the Lie group $\SL(n)$, the exponential map is a diffeomorphism from a neighborhood $U_0$ of $0$ in $\M{n}{n}_{\rm dev}$ to a neighborhood $U_{\Id}$ of $\Id$ in $\SL(n)$ whose inverse is the matrix logarithm. Since these mappings preserve symmetry, for $\eps$ so small that $ X_{\eps} := \Id + \eps Y \in U_{\Id}$ there is a unique $Z_{\eps} \in U_0 \cap \M{n}{n}_{\rm ils}$ such that 
\[ ( \det X_{\eps} )^{-1/n} X_{\eps} = e^{ Z_{\eps} }. \] 
Since on the other hand
\begin{align*}
  ( \det X_{\eps} )^{-1/n} X_{\eps} 
  &= \big( 1 - \tfrac{\eps}{n} \tr Y + O(\eps^2) \big) ( \Id + \eps Y ) \\
  &= \Id + \eps Y_{\rm dev} + O(\eps^2) 
   = e^{\eps Y_{\rm dev}} + O(\eps^2),  
\end{align*}
we have that 
\[ Z_{\eps} = \eps Y_{\rm dev} + O(\eps^2). \] 
Also note that 
\begin{align*}
  ( \det X_{\eps} )^{-1/n} X_{\eps} - X_{\eps}
  &= \big( - \tfrac{\eps}{n} \tr Y + O(\eps^2) \big) ( \Id + \eps Y ) \\
  &= - \tfrac{\eps}{n} ( \tr Y ) \Id + O(\eps^2). 
\end{align*}
In both estimates the error terms only depend on $r$. 
It follows that 
\begin{align*} 
\widetilde{V}_{\eps,\kappa}(x,Y)
& = V_{\eps} (x,\eps^{-1} Z_{\eps}) - V(x,\eps^{-1} Z_{\eps}) 
+ V(x,\eps^{-1} Z_{\eps})  + \kappa n^{-p/2}| \tr Y + O(\eps) |^p \\
& = V_{\eps} (x,\eps^{-1} Z_{\eps}) - V(x,\eps^{-1} Z_{\eps}) 
+ V(x,Y_{\rm dev} + O(\eps)) + \kappa n^{-p/2}| \tr Y + O(\eps) |^p, 
\end{align*}
and with the help of the Scorza-Dragoni theorem we obtain  
\[ \lim_{\eps \to 0}\int_{\Omega} \sup_{|Y| \le r} | \widetilde{V}_{\eps,\kappa}(x,Y) - V_{\kappa}(x,Y)| \x 
   = 0. \]
Now note that by the growth assumptions on $V$ and the identity $\dist(X_{\eps}, \SO(n)) = \eps Y$, which holds due to the symmetry of $Y$, 
\begin{align*} 
  V_{\kappa}(x,Y) 
  &\le \beta|Y_{\dev}|^p + \beta + \kappa n^{-p/2}|\tr Y|^p \\ 
  &\le \frac{1}{\eps^p} \kappa^2 \dist^p ( X_{\eps}, \SO(n) ) + \kappa^2  
\end{align*}
for sufficiently large $\kappa$. This shows that $|V_{\eps,\kappa}(x,Y) - V_{\kappa}(x,Y)| \le |\widetilde{V}_{\eps,\kappa}(x,Y) - V_{\kappa}(x,Y)|$ for a.e.\ $x \in \Omega$ and $Y \in \M{n}{n}_{\sm}$ with $|Y| \le r$. Thus we also have 
\[ \lim_{\eps \to 0}\int_{\Omega} \sup_{|Y| \le r} |V_{\eps,\kappa}(x,Y) - V_{\kappa}(x,Y)| \x 
   = 0. \]
We define the functionals 
\[ \F_{ \eps, \kappa }(v) 
:= \left\{ \begin{array}{cl}
\int_{ \Omega } V_{ \eps, \kappa }( x ,\nabla v(x) ) \x, & v \in g_{\eps} + W^{1,p}_{ \partial \Omega_{*}}( \Omega ; \R^{n} ), \\
\i, & \mbox{else},
\end{array} \right. \]
and
\[ \F_{ {\rm rel}, \kappa }(v) 
:= \left\{ \begin{array}{cl}
\int_{ \Omega } ( V_{\kappa} )^{\qc}( x , \E v(x) ) \x, & v \in g + W^{1,p}_{ \partial \Omega_{*}}( \Omega ; \R^{n} ), \\
\i, & \mbox{else},
\end{array} \right. \]
and observe that the functionals $\F_{ \eps, \kappa }$ satisfy a uniform $p$-G{\r a}rding inequality: 
\[ \F_{ \eps, \kappa }(v) 
   \ge \int_{\Omega} \left( \frac{c}{\eps^p} \dist^p ( \Id + \eps \nabla v(x) , \SO(n) ) - \beta \right) \x 
   \ge c' \| \nabla v \|_{L^p}^p - C \| v \|_{L^p}^p \]
for suitable constants $c, c', C > 0$ independent of $\eps$. Here the last step follows from a nonlinear version of Korn's inequality which is a consequence of the geometric rigidity theorem proved in \cite{FJM:02} (for $p = 2$ and extended to general $p$ in \cite{ContiSchweizer}).\footnote{There exists a constant 
$ C $, depending only on $\Omega$ and $p$, such that for every $ v \in W^{1,p}( \Omega , \R^{n} ) $
\[ \| \nabla v \|_{ L^{p}( \Omega , \M{n}{n} ) } 
   \le C \big( \| \dist \big( \Id + \nabla v , \SO(n) \big) \|_{ L^{p}( \Omega ) } 
   + \| v \|_{ L^{p}( \Omega , \R^{n} ) } \big). \]
For a proof see, e.g., Theorem A.8 in \cite{JesenkoSchmidt:14}.
}
By construction the densities $W_{\eps,\kappa}$ also satisfy a uniform $p$-growth assumption from above: 
\[ - \beta 
   \le V_{\eps,\kappa}(x, Y) 
   \le \frac{1}{\eps^p} \big( \kappa^2 \dist^p ( \Id + \eps Y, \SO(n) ) + \kappa^2 \eps^p \big) 
   \le \kappa^2 |Y|^p + \kappa^2. \] 
Finally we remark that standard relaxation results (Theorem~\ref{theo:lsc}, Korn's inequality and Theorem 9.8 in \cite{Dacorogna}) identify $\F_{ {\rm rel}, \kappa }$ as the $L^p$-lower semicontinuous envelope of 
\[ \F_{ \kappa }(v) 
:= \left\{ \begin{array}{cl}
\int_{ \Omega } V_{\kappa} ( x , \E v(x) ) \x, & v \in g + W^{1,p}_{ \partial \Omega_{*}}( \Omega ; \R^{n} ), \\
\i, & \mbox{else}.
\end{array} \right. \]
Suppose now that $u_{\eps} \to u$ in $L^p(\Omega)$. Without loss of generality we may assume $ \F_{ \eps }(u_{\eps}) < \i $ for all $\eps$, i.e., $u_{\eps} \in g_{\eps} + W^{1,p}_{ \partial \Omega_{*} }( \Omega ; \R^{n} )$ and $ \Id + \eps \nabla u_{\eps} \in \SL(n)$ for all $\eps$. By the uniform $p$-G{\r a}rding inequality we then also have $u_{\eps} \weakly u$ in $W^{1,p}( \Omega ; \R^{n} )$ and $u \in g + W^{1,p}_{ \partial \Omega_{*}}( \Omega ; \R^{n} )$. 
As a consequence to the above we may invoke Theorem 3.2 in \cite{JesenkoSchmidt:14} to obtain 
\[ \liminf_{\eps \to 0} \F_{ \eps }(u_{\eps}) 
   = \liminf_{\eps \to 0} \F_{ \eps, \kappa }(u_{\eps}) 
   \ge \F_{{\rm rel}, \kappa} (u) \]
for every $ \kappa $ sufficiently large. From Lemma~\ref{lemma:iqc:2}, it follows
\[ V_{\kappa}^{\qc}(x,X)
\nearrow
\left\{ \begin{array}{cl}
V^{\iqc}(x,X), & \tr X = 0, \\
\i, & \mbox{else}.
\end{array} \right. \]
By the monotone convergence theorem
$ \F_{ {\rm rel}, \kappa } \nearrow \F_{ {\rm rel} } $
as $\kappa \nearrow \infty$, and thus
\[ \liminf_{\eps \to 0} \F_{ \eps }(u_{\eps}) 
   \ge \F_{{\rm rel}} (u). \qedhere \] 
\end{proof} 
%
%%%%%%%%%%%%%%%%%%%%%%%%%%%%%%%%%%%%%%%%%%%%%%%%%%%%%%%%%%%%%%%%%%%%%%%%%%%%%%%%%%%%%%%%%%%%%%%%%%%%%%%%%%%%%%%%%%%%%%%%%
\subsection{Recovery sequence}
%%%%%%%%%%%%%%%%%%%%%%%%%%%%%%%%%%%%%%%%%%%%%%%%%%%%%%%%%%%%%%%%%%%%%%%%%%%%%%%%%%%%%%%%%%%%%%%%%%%%%%%%%%%%%%%%%%%%%%%%%
For the construction of recovery sequences, we will need two auxiliary results.
First we approximate solenoidal fields by fields corresponding to displacements in the non-linear theory of incompressible materials.
%------------------------------------------------------------------------------------------------------------------------
\begin{lemma}
\label{lemma:approx-div-free}
Suppose $ u \in C^{\i}_{c}( \R^{n} ; \R^{n} ) $ is a solenoidal vector field.
There exist $ u_{\eps} \in C^{\i}_{c}( \R^{n} ; \R^{n} ) $ such that 
\[ \supp u_{\eps} \subset \supp u, \quad
\forall x \in \R^{n}: \Id + \eps \nabla u_{\eps}(x) \in \SL(n) 
\quad \mbox{and} \quad
\| u_{ \eps } - u \|_{ C^k } \to 0 \]
for any $k \in \N$ as $\eps \searrow 0$.
\end{lemma}
\begin{proof}
The initial value problem 
\[ \partial_t y(t, x) = u(y(t,x)), \quad y(0, x) = x \]
has a unique smooth solution $ y : \R \times \R^{n} \to \R^{n} $.
Clearly, if $ u(x) = 0 $, then $ y(\cdot,x) = x $.
Since $\DIV u = 0$, the induced flow is volume preserving, and we have 
\[ \det \nabla_{\! x} y( t , x ) 
   = \det \nabla_{\! x} y( 0 , x ) 
   = 1 \] 
for all $ (t, x) \in \R \times \R^{n} $. 
\item
Now define $ u_{\eps} \in C^{\i}( \R^{n} ; \R^{n} ) $ by $ u_{\eps}(x) := \frac{1}{ \eps }( y( \eps , x ) - x ) $
so that $\Id + \eps \nabla u_{\eps}(x) = \nabla_{\! x} y( \eps , x ) \in \SL(n)$. 
Since $y$ is smooth, we also have 
\[ u_{\eps} \to \partial_t y( 0 , \cdot ) = u 
   \quad\mbox{and}\quad 
   \nabla_{\! x}^k u_{\eps} 
   \to \partial_t \nabla_{\! x}^k y( 0 , \cdot ) 
   = \nabla_{\! x}^k \partial_t y( 0 , \cdot) 
   = \nabla^k u \] 
uniformly on $\R^{n}$ as claimed. 
\end{proof} 
%------------------------------------------------------------------------------------------------------------------------

Moreover, we will need a result on extensions of solenoidal fields with a given zero set.
The result was shown in \cite{Braides} (see Proposition 3.8) for the case $ D = \emptyset $.
We adapt that proof to our more general setting.
\begin{prop}
\label{prop:ext-div-free-bd}
Suppose $ \Omega , \Omega' \subset \R^{n} $ are bounded Lipschitz domains
and $ D \subset \R^{n} $ a closed set such that $ D \subset \overline{ \Omega } \subset \Omega' $. 
Then for any solenoidal $ u \in W_{D}^{1,p}( \Omega ; \R^{n} ) $ 
there exists a function $ v \in W_{D \cup \partial \Omega'}^{1,p}( \Omega' ; \R^{n} ) $ such that 
$ u = v $ a.e.~in $ \Omega $ and $ \DIV v = 0 $ in $ \Omega' $.
\end{prop}
\begin{proof}
We extend the function $u$ to $ \Omega' $ by 
employing the extension operator 
\[ E : W^{1,p}_{D}( \Omega ; \R^{n} ) \to W^{1,p}_{D}( \R^{n} ; \R^{n} ) \]
from Theorem~\ref{theo:extension} and multiplying the new function with a suitable cut-off function so as to 
get a function $ w \in W^{1,p}_{D \cup \partial \Omega'}( \Omega' ; \R^{n} ) $ with $ w|_{\Omega} = u $ a.e.
Then
\[ 0 = \int_{ \Omega' } \DIV w(x) \x
= \int_{ \Omega' \setminus \overline{ \Omega } } \DIV w(x) \x. \]
Since $ \Omega' \setminus \overline{\Omega} $ is a Lipschitz domain, by Theorem~\ref{theo:Bogovskii}, 
there exists a function $ w' \in W^{1,p}_{0}( \Omega' \setminus \overline{\Omega} ; \R^{n} ) $ such that 
$ \DIV w' = \DIV w $ on $ \Omega' \setminus \overline{\Omega} $.
Then, regarding $w'$ as a function defined on $ \R^{n} $, the function $ v := w - w' $ has the desired properties.
\end{proof}
%------------------------------------------------------------------------------------------------------------------------
\begin{proof}[Proof of existence of a recovery sequence in Theorem~\ref{theo:gamma}.]
Let us take any solenoidal function $ u \in g + W^{1,p}_{ \partial \Omega_{*} }(\Omega, \R^{n} ) $, and let $ \eta > 0 $ be arbitrary. We define $V_{g} : \Omega \times \M{n}{n} \to \R$ by $V_{g}(x, Z) := V(x, Z + \E g(x))$. 
Since $g \in W^{1,\infty}(\Omega; \R^{n})$, this function satisfies the assumptions of Theorem~\ref{theo:iqc}. By Lemmas~\ref{lemma:iqc:1} and \ref{lemma:iqc:2}, there exists a solenoidal $ v^{\eta} \in u - g + C_{c}^{\i}( \Omega ; \R^{n} ) $ such that $\| v^{\eta} - u + g \|_{ L^{p}(\Omega) } < \eta$ and 
\[ \int_{ \Omega } V_{g}( x , \E v^{ \eta }(x) ) \x 
   \le \int_{ \Omega } V_{g}^{ \iqc } ( x , \E u(x) - \E g(x)) \x + \eta
	 = \int_{ \Omega } V^{ \iqc }( x , \E u(x) ) \x + \eta, \]
where we have used that $V_{g}^{ \iqc }( x , X - \E g(x)) = V^{ \iqc }( x , X ) $ for all $X \in \M{n}{n}_{\dev}$ and $a.e.~x\in\Omega$ by Definition~\ref{def:iqc}.  Choose some open ball $B$ with $ \Omega \subset \subset B $. 
By Proposition~\ref{prop:ext-div-free-bd} we may extend $ v^{\eta} $ 
to some $ w^{\eta} \in W^{1,p}_{ \partial \Omega_{*} \cup \partial B }( B, \R^{n} ) $ with $ \DIV w^{\eta} = 0 $ on $ B $. 
\item 
By Theorem~\ref{theo:density-div-free}, solenoidal functions from $ C_{c}^{\i}( B \setminus \partial \Omega_{*} ) $ are dense in $ \{ v \in W^{1,p}_{ \partial \Omega_{*} \cup \partial B }( B, \R^{n} ) : \DIV v = 0 \} $.
Thus we obtain a solenoidal $ \phi^{\eta} \in C_{c}^{\infty}( B \setminus \partial \Omega_{*} ; \R^{n} ) $ such that $ \| \phi^{\eta} - w^{\eta} \|_{ W^{1,p}(B) } < \eta $.
Our assumptions on $V$ and H{\"o}lder's inequality yield 
\begin{align*}
  &\int_{ \Omega } \big| V_{g}( x , \E w^{\eta}(x) ) - V_{g}( x , \E \phi^{\eta} (x) ) \big| \x \\ 
  &~~\le C \big( 1 + \| \E w^{\eta} \|_{ L^{p}( \Omega ) }^{p-1} + \| \E \phi^{\eta} \|_{ L^{p}( \Omega ) }^{p-1} \big) 
        \| \E ( w^{\eta} - \phi^{\eta} ) \|_{ L^{p}( \Omega ) }, 
\end{align*}
where $\| \E \phi^{\eta} \|_{ L^{p}( \Omega ) } \le \| \E w^{\eta} \|_{ L^{p}( \Omega ) } + \eta$ so that, upon choosing $\phi^{\eta}$ sufficiently close to $w^{\eta}$, we can also achieve that 
\[ \int_{ \Omega } \big| V_{g}( x , \E w^{\eta}(x) ) - V_{g}( x , \E \phi^{\eta} (x) ) \big| \x
   < \eta, \]
and thus 
\[ \int_{ \Omega } V_{g}( x , \E \phi^{ \eta }(x) ) \x \le \int_{ \Omega } V^{ \iqc }( x , \E u(x) ) \x + 2 \eta. \]
\item 
According to Lemma~\ref{lemma:approx-div-free}, there exist $ \phi^{\eta}_{\eps} \in C^{\infty}_{c}( B \setminus \partial \Omega_{*}; \R^{n} ) $ 
such that 
\[ \forall x \in \R^{n}: \Id + \eps \nabla \phi^{\eta}_{\eps}(x) \in \SL(n)  
\quad \mbox{and} \quad
\| \phi^{\eta}_{\eps} - \phi^{\eta} \|_{ C^1_{c}( \R^{n} ) } \to 0. \]
Let us extend $g_{\eps}$ and $g$ to functions in $W^{1,\infty}(B; \R^n)$ with the help of Theorem~\ref{theo:extension}. Note that still $g_{\eps} \to g$ in $W^{1,\infty}(B; \R^n)$. We define functions 
$ u^{\eta}_{\eps} : \Omega \to \R^{n} $ by 
\[ u^{\eta}_{\eps}(x)
   := \phi^{\eta}_{\eps}(x) + g_{\eps}(x + \eps \phi^{\eta}_{\eps}(x))
\]
for $\eps$ small enough. Since the inner functions $ \id + \eps \phi^{\eta}_{\eps} $ are diffeomorphisms for $\eps$ sufficiently small, we have $ u^{\eta}_{\eps} \in W^{1,\i}( \Omega ; \R^n) $ and 
\[ \nabla u^{\eta}_{\eps}(x) = \nabla \phi^{\eta}_{\eps}(x) + \nabla g_{\eps }( x + \eps \phi^{\eta}_{\eps}(x) )( \Id + \eps \nabla \phi^{\eta}_{\eps} ), \]
see, e.g.~\cite[Lemma 2.3.2]{Necas}. 
Therefore, $ \| u^{\eta}_{\eps} \|_{ W^{1,\i}( \Omega )} < C $. We also have $ u^{\eta}_{\eps} \in g_{\eps} + W^{1,p}_{ \partial \Omega_{*} }(\Omega, \R^{n} )$. Namely, by approximating $g_{\eps}$ in $W^{1,p}(B; \R^n)$ with a sequence of smooth functions $ g_{\eps}^{k} \in C_{c}^{\i}( \R^{n} ; \R^{n} ) $, we obtain functions $ \phi^{\eta}_{\eps} + g_{\eps}^{k} \circ (\id + \eps \phi^{\eta}_{\eps}) - g_{\eps}^{k} $ lying in $ C_{c}^{\i}( \R^{n} \setminus \partial \Omega_{*}; \R^{n} ) $ and approximating $ u^{\eta}_{\eps} - g_{\eps} $ in $W^{1,p}( \Omega ; \R^n)$. Moreover, $ u^{\eta}_{\eps} $ is in the domain of $ \F_{\eps} $ since for
\[ y^{\eta}_{\eps} 
   := \id + \eps u^{\eta}_{\eps}
   = (\id + \eps g_{\eps}) \circ (\id + \eps \phi^{\eta}_{\eps})
\]
we have
\[  \nabla y^{\eta}_{\eps} 
= (\Id + \eps \nabla g_{\eps}(\cdot + \eps \phi^{\eta}_{\eps})) (\Id + \eps \nabla \phi^{\eta}_{\eps}) \]
and in particular $ \det \nabla y^{\eta}_{\eps} = 1 $ a.e. 

As $\phi^{\eta}_{\eps} \to \phi^{\eta}$ and $g_{\eps} \to g$ uniformly and $g$ is continuous, we have that $u^{\eta}_{\eps} \to \phi^{\eta} + g$ in $ L^{p}( \Omega ; \R^{n} ) $, where $\| \phi^{\eta} + g - u \|_{ L^p(\Omega) } < 2 \eta$. Moreover, we claim that 
\[ \nabla u^{\eta}_{\eps} 
   \to \nabla \phi^{\eta} + \nabla g \qquad {\mbox{in } L^p(\Omega; \M{n}{n})}. 
\] 
We first note that $\nabla \phi^{\eta}_{\eps} \to \nabla \phi^{\eta}$ uniformly and 
\[ \| \eps \nabla g_{\eps}(\cdot + \eps \phi^{\eta}_{\eps}) \nabla \phi^{\eta}_{\eps} \|_{ L^{\infty}(\Omega) } 
   \le C \eps \to 0 \]  
as $\eps \to 0$. Hence, it remains to show $ \nabla g_{\eps}(\cdot + \eps \phi^{\eta}_{\eps}) \to \nabla g $ in $ L^{p}( \Omega ) $.

For given $\delta > 0$ we can choose a continuous $H \in C(B; \M{n}{n})$ such that $\| \nabla  g - H \|_{L^p{(B)}} < \delta$. Then 
\begin{align*}
  &\| \nabla g_{\eps}(\cdot + \eps \phi^{\eta}_{\eps}) - \nabla g \|_{L^p(\Omega)} \\ 
	&~~\le \| \nabla g_{\eps}(\cdot + \eps \phi^{\eta}_{\eps}) - \nabla g(\cdot + \eps \phi^{\eta}_{\eps}) \|_{L^p(\Omega)} 
	   + \| \nabla g(\cdot + \eps \phi^{\eta}_{\eps}) - H(\cdot + \eps \phi^{\eta}_{\eps}) \|_{L^p(\Omega)} \\
	&\qquad + \| H(\cdot + \eps \phi^{\eta}_{\eps}) - H \|_{L^p(\Omega)} 
		 + \| H - \nabla g \|_{L^p(\Omega)}.  
\end{align*}
Here the first and the third term on the right hand side tend to $0$ as $\eps \to 0$, while the fourth term is bounded by $\delta$. Since $\varphi_{\eps} := \id + \eps \phi^{\eta}_{\eps}$ is a diffeomorphism from $\Omega$ to $\varphi_{\eps}(\Omega)$ with $\| \det \nabla \varphi_{\eps}^{-1} \|_{L^{\infty}} \le 1 + C \eps$, for the second term we obtain 
\begin{align*}
 \int_{\Omega} | \nabla g(x + \eps \phi^{\eta}_{\eps}(x)) - H(x + \eps \phi^{\eta}_{\eps}(x)) |^p \x 
	 &\le (1 + C \eps) \int_{B} | \nabla g(y) - H(y) |^p \y 
	 \le 2^p \delta^p 
\end{align*}
if $\eps$ is small enough. Collecting terms we obtain 
\[ \| \nabla g_{\eps}(\cdot + \eps \phi^{\eta}_{\eps}) - \nabla g \|_{L^p(\Omega)} 
	< 4 \delta 
\] 
for $\eps$ sufficiently small, which concludes the proof of the claim since $\delta > 0$ was arbitrary. 

The polar decomposition yields 
\[ \Id + \eps \nabla u^{\eta}_{\eps} 
   = R^{\eta}_{\eps} e^{\eps Z^{\eta}_{\eps}} \] 
for some mappings $R^{\eta}_{\eps} : \R^{n} \to \SO(n)$ 
and $Z^{\eta}_{\eps} : \R^{n} \to \M{n}{n}_{\rm ils}$.
Thus, we may rewrite
\[ \frac{1}{ \eps^{p} } W_{ \eps }( x , \Id + \eps \nabla u^{\eta}_{\eps} )
= \frac{1}{ \eps^{p} } W_{ \eps }( x , R^{\eta}_{\eps} e^{\eps Z^{\eta}_{\eps}} )
= \frac{1}{ \eps^{p} } W_{ \eps }( x , e^{\eps Z^{\eta}_{\eps}} ) 
= V_{ \eps }( x , Z^{\eta}_{\eps} ). \]
From
\[ e^{\eps Z^{\eta}_{\eps}} 
= \sqrt{ ( \Id + \eps \nabla u^{\eta}_{\eps} )^{T}( \Id + \eps \nabla u^{\eta}_{\eps} ) } 
= \Id + \eps \E u^{\eta}_{\eps} + O( \eps^{2} ) \]
and the above claim it follows that $Z^{\eta}_{\eps} \to \E \phi^{\eta} + \E g$ in $L^p(\Omega)$. Now  with $ r := \sup_{\eps} \max | Z^{\eta}_{\eps} | $ we have 
\[ \limsup_{\eps\to 0} \int_{\Omega} |V_{\eps}(x, Z^{\eta}_{\eps}(x) ) - V(x, Z^{\eta}_{\eps}(x) )| \x
\le \lim_{\eps\to 0} \int_{\Omega} \sup_{|X| \le r} |V_{\eps}(x,X) - V(x,X)| \x = 0. \] 

Moreover, (L1) and H{\"o}lder's inequality yield 
\begin{align*}
  &\int_{ \Omega } \big| V( x , Z^{\eta}_{\eps}(x) ) - V_{g} ( x , \E \phi^{ \eta }(x) ) \big| \x \\ 
  &~~\le C \big( 1 + \| Z^{\eta}_{\eps} \|_{L^{p}(\Omega)}^{p-1} + \| \E \phi^{\eta} + \E g \|_{L^{p}(\Omega)}^{p-1} \big) 
        \| Z^{\eta}_{\eps} - \E ( g + \phi^{\eta} ) \|_{L^{p}(\Omega)} \to 0.
\end{align*}

Hence, 
\begin{align*} 
\lim_{ \eps \to 0 } \frac{1}{ \eps^{p} } \int_{ \Omega } W_{ \eps }( x , \Id + \eps \nabla u^{\eta}_{\eps}(x) ) \x 
  &= \lim_{ \eps \to 0 } \int_{ \Omega } V_{ \eps }( x , Z^{\eta}_{\eps}(x) ) \x \\ 
  &= \lim_{ \eps \to 0 } \int_{ \Omega } V( x , Z^{\eta}_{\eps}(x) ) \x 
   = \int_{ \Omega } V_{g} \big( x , \E \phi^{ \eta } (x) \big) \x. 
\end{align*}

Summarizing, we have seen that $u^{\eta}_{\eps} \to \phi^{\eta} + g$ in $L^p(\Omega; \R^n)$, where $\| \phi^{\eta} + g - u \|_{L^p(\Omega)} < 2 \eta$, and 
\begin{align*} 
\lim_{ \eps \to 0 } \F_{\eps}(u^{\eta}_{\eps}) 
  &\le \int_{ \Omega } V^{ \iqc }( x , \E u(x) ) \x + 2 \eta 
	=\F_{{\rm rel}} (u) + 2 \eta. 
\end{align*}
Finally, as $\eta$ was arbitrary, we find a recovery sequence by passing to a suitable diagonal sequence. For $u \notin g + W^{1,p}_{ \partial \Omega_{*} }(\Omega, \R^{n} )$ the existence of a recovery sequence is trivial. 
\end{proof} 

%%%%%%%%%%%%%%%%%%%%%%%%%%%%%%%%%%%%%%%%%%%%%%%%%%%%%%%%%%%%%%%%%%%%%%%%%%%%%%%%%%%%%%%%%%%%%%%%%%%%%%%%%%%%%%%%%%%%%%%%%
\subsection{Compactness}
\label{subsect:Compactness}
%%%%%%%%%%%%%%%%%%%%%%%%%%%%%%%%%%%%%%%%%%%%%%%%%%%%%%%%%%%%%%%%%%%%%%%%%%%%%%%%%%%%%%%%%%%%%%%%%%%%%%%%%%%%%%%%%%%%%%%%%

As we will see, compactness modulo rigid motions follows directly with the help of a geometric rigidity argument. Full compactness for the case with Dirichlet boundary conditions is more involved. This has been first shown in \cite{DalMasoNegriPercivale} for $p=2$ and single well energies. Although the proof given there can be adapted to our more general setup in a straightforward way, we include a more direct argument below. 

The following lemma is needed for Dirichlet boundary conditions. It provides a simplified alternative to Lemma 3.3 and Proposition 3.4 in \cite{DalMasoNegriPercivale}. 

\begin{lemma}\label{lemma:bdry-rot}
Suppose that $S \subset \R^n$ is such that ${\cal H}^{n-1}(S) > 0$. Then there is a constant $C > 0$ such that 
\[ |R - \Id| + |c| 
   \le C \int_{S} |Rx - x - c| \ d{\cal H}^{n-1}(x) \] 
for all $R \in \SO(n)$ and $c \in \R^n$. 
\end{lemma} 

\begin{proof}
If not, then for any $k \in \N$ we find $R_k \in \SO(n)$ and $c_k \in \R^n$ such that 
\[ |R_k - \Id| + |c_k| 
   > k \int_{S} |R_k x - x - c_k| \ d{\cal H}^{n-1}(x). \] 
We first observe that $(c_k)$ must be bounded because of 
\begin{align*} 
  {\cal H}^{n-1}(S) | c_k | 
  &\le \int_{S} |R_k x - x - c_k| \ d{\cal H}^{n-1}(x) 
      + |R_k - \Id| \int_{S} |x| \ d{\cal H}^{n-1}(x) \\ 
  &\le C |R_k - \Id| + \frac{1}{k} |c_k|.  
\end{align*}
Passing to a subsequence, we can now assume that $R_k \to R$ and $c_k \to c$ for some $R \in \SO(n)$ and $c \in \R^n$. Sending $k \to \infty$, we obtain  
\[ \int_{S} |Rx - x - c| \ d{\cal H}^{n-1}(x) = 0, \] 
so that $Rx - x = c$ for ${\cal H}^{n-1}$-a.e.\ $x \in S$. Let $x_0$ be one of these $x$ and set $S_0 = S-x_0$. Then for ${\cal H}^{n-1}$-a.e.\ $x \in S_0$ we have $Rx - x = 0$. Now ${\cal H}^{n-1}(S_0) > 0$ implies that there are $n-1$ linearly independent eigenvectors corresponding to the eigenvalue $1$ of $R$. Since $R \in \SO(n)$, we then must have $R = \Id$ und thus also $c = 0$. 

Set $A_k = R_k - \Id$ and choose a subsequence such that $\frac{A_k}{|A_k| + |c_k|} \to A$ and $\frac{c_k}{|A_k| + |c_k|} \to c'$ for suitable $A \in \M{n}{n}$ and $c' \in \R^n$. Note that then $A$ is antisymmetric. Letting $k \to \infty$ in 
\[ 1 
   \ge k \int_{S} \Big| \frac{A_k x}{|A_k| + |c_k|} - \frac{c_k}{|A_k| + |c_k|} \Big| \ d{\cal H}^{n-1}(x) \] 
gives 
\[ \int_{S} | Ax - c' | \ d{\cal H}^{n-1}(x) = 0. \] 
Analogously as above it follows that there are $n-1$ linearly independent eigenvectors corresponding to the eigenvalue $0$ of $A$. Since $A$ is antisymmetric it cannot be of rank $1$, so we obtain $A = 0$ and then also $c' = 0$. This contradicts 
\[ 1 
   = \frac{|A_k|}{|A_k| + |c_k|} + \frac{|c_k|}{|A_k| + |c_k|} 
   \to |A| + |c'| 
   = 0. \qedhere \] 
\end{proof}

\begin{proof}[Proof of Theorem~\ref{theo:compactness}.] 
Suppose $ y_{\eps} \in W^{1,p}( \Omega ; \R^{n} ) $ satisfies $ \En_{\eps}( y_{\eps} ) \le C$. It follows from geometric rigidity \cite{FJM:06,ContiSchweizer} that there exists a constant $c > 0$ such that 
\begin{align*}
  C 
  &\ge \frac{1}{ \eps^{p} } \int_{ \Omega } W_{\eps}( x , \nabla y_{\eps}(x) ) - \frac{1}{ \eps } \int_{ \Omega } \tilde{\ell}_{\eps}(x) \cdot y_{\eps}(x) \x \\ 
  &\ge \frac{1}{ \eps^{p} } \int_{ \Omega } \left( \alpha \dist^p( \nabla y_{\eps}(x), \SO(n) ) - \beta \eps^p \right) \x - \frac{1}{ \eps } \int_{ \Omega } \tilde{\ell}_{\eps}(x) \cdot y_{\eps}(x) \x \\ 
  &\ge \int_{ \Omega } c \Big| \frac{\nabla y_{\eps}(x) - R_{\eps}}{\eps} \Big|^p \x - \beta | \Omega |  - \int_{ \Omega } \tilde{\ell}_{\eps}(x) \cdot \frac{y_{\eps}(x) - R_{\eps}(x + c_{\eps})}{ \eps } \x 
\end{align*}

for suitable rotations $R_{\eps} \in \SO(n)$ and arbitrary $c_{\eps} \in \R^n$, where we have also used the invariance properties of $\ell_{\eps}$. Choose $ c_{\eps} := \frac{1}{ |\Omega| } \int_{\Omega} ( R_{\eps}^{T} y_{\eps}(x) - x ) \x $. For
\[ u_{\eps}(x) := \frac{R_{\eps}^T y_{\eps}(x) - x - c_{\eps}}{\eps}, \]  
we obtain 
\[ C + |\Omega| \beta
   \ge c \| \nabla u_{\eps}(x) \|_{L^p}^p - \| \tilde{\ell}_{\eps} \|_{L^{p'}} \| u_{\eps} \|_{L^p} \] 
so that Poincar\'{e} inequality shows that $(u_{\eps})$ is bounded in $W^{1,p}(\Omega; \R^n)$, and the assertion of (a) follows.  

In order to prove (b) we first apply the assertion of (a) to find $\tilde{R}_{\eps} \in \SO(n)$ and $\tilde{c}_{\eps} \in \R^n$ such that, for a subsequence, $\tilde{u}_{\eps} := \frac{1}{\eps}(\tilde{R}_{\eps}^T y_{\eps} - \id - \tilde{c}_{\eps}) \weakly \tilde{u}$ in $W^{1,p}( \Omega ; \R^{n} )$. By the trace theorem, the restriction of $\tilde{u}_{\eps}$ to $\partial \Omega_{*}$ is bounded in $L^1(\partial \Omega_{*}; \R^n)$ and satisfies $\tilde{u}_{\eps}(x) = \frac{1}{\eps}(\tilde{R}_{\eps}^T w_{\eps}(x) - x - \tilde{c}_{\eps}) = \frac{1}{\eps}(\tilde{R}_{\eps}^T x - x - \tilde{c}_{\eps}) + \tilde{R}_{\eps}^T g_{\eps}(x)$ 
for ${\cal H}^{n-1}$-a.e.\ $x \in \partial \Omega_{*}$. Lemma \ref{lemma:bdry-rot} now implies 
\[
   |\tilde{R}_{\eps}^T - \Id| + |\tilde{c}_{\eps}| 
   \le C \eps \int_{\partial \Omega_{*}} ( |\tilde{u}_{\eps}(x)| + |g_{\eps}(x)| ) \ d{\cal H}^{n-1}(x) 
	 \le C \eps 
\] 
so that, upon passing to a further subsequence, we have $\frac{1}{\eps}(\tilde{R}_{\eps} - \Id) \to A$ and $\frac{\tilde{R}_{\eps}\tilde{c}_{\eps}}{\eps} \to c$ for some $A \in \M{n}{n}$ and $c \in \R^n$. We thus get 
\[ \frac{y_{\eps}(x) - x}{\eps} 
   = \tilde{R}_{\eps} \tilde{u}_{\eps}(x) + \frac{(\tilde{R}_{\eps} - \Id) x}{\eps} + \frac{\tilde{R}_{\eps}\tilde{c}_{\eps}}{\eps}
   \weakly \tilde{u}(x) + Ax + c =: u(x) 
\] 
and $y_{\eps} \in w_{\eps} + W^{1,p}_{ \partial \Omega_{*} }( \Omega ; \R^{n} ) $  for all $\eps$ implies $u \in g + W^{1,p}_{ \partial \Omega_{*} }( \Omega ; \R^{n} ) $.

\end{proof}

%%%%%%%%%%%%%%%%%%%%%%%%%%%%%%%%%%%%%%%%%%%%%%%%%%%%%%%%%%%%%%%%%%%%%%%%%%%%%%%%%%%%%%%%%%%%%%%%%%%%%%%%%%%%%%%%%%%%%%%%%
\subsection{Low energy sequences}
%%%%%%%%%%%%%%%%%%%%%%%%%%%%%%%%%%%%%%%%%%%%%%%%%%%%%%%%%%%%%%%%%%%%%%%%%%%%%%%%%%%%%%%%%%%%%%%%%%%%%%%%%%%%%%%%%%%%%%%%%

The following proof follows along standard arguments in the theory of $\Gamma$-convegence with only mild extra effort to treat the $R$-dependent limit. We include it for the sake of completeness. 
 
\begin{proof}[Proof of Corollary~\ref{cor:low-energy-seq}.] 
Denoting by $\eps Z_{\eps}(x) \in \M{n}{n}_{\dev}$ the matrix logarithm of $\nabla w_{\eps} = \Id + \eps \nabla g_{\eps}(x)$ so that $e^{\eps Z_{\eps}} = \Id + \eps \nabla g_{\eps}$ with $\| Z_{\eps} \|_{L^{\infty}} \le C$, we see that 
\[ \inf_{w \in L^p(\Omega;\R^n)} \En_{\eps}(w) 
   \le \En_{\eps}(w_{\eps}) = \int_{\Omega} V_{\eps}(x,Z_{\eps}(x)) \x 
	    - \int_{ \Omega } \tilde{\ell}_{\eps}(x) \cdot g_{\eps}(x) \x
	 \le C\] 
because of condition (C). This shows that for arbitrary low energy sequence $ (y_{\eps}) $, $ \En_{\eps}( y_{\eps} ) $ is bounded. From Theorem~\ref{theo:compactness} we obtain $R_{\eps} \in \SO(n)$ and $c_{\eps} \in \R^n$ such that, for a subsequence, $u_{\eps} := \frac{R_{\eps}^T y_{\eps} - \id - c_{\eps}}{\eps} \weakly u$ in $W^{1,p}( \Omega ; \R^{n} )$, where in the case (b) one may choose $R_{\eps} = R = \Id$ and $c_{\eps} = 0$. Passing to a further subsequence, we also have $R_{\eps} \to R$ for some $R \in \SO(n)$. 

Take any $v \in g + W_{ \partial \Omega_{*} }^{1,p}( \Omega ; \R^{n} ) $ with $ \DIV v = 0$ and choose a recovery sequence $(v_{\eps})$ for $v$ with respect to the $\Gamma$-convergence of $\F_{\eps} \to \F_{\rm rel}$. Let $R' \in \SO(n)$ be arbitrary if $\partial \Omega_{*} = \emptyset$ and $R' = R = \Id$ if ${\cal H}^{n-1}(\partial \Omega_{*}) > 0$. Set $ z_{\eps}(x) := R'(x + \eps v_{\eps}(x))$. 
Then 
\begin{align*}
  \F_{\rm rel}(u) - \int_{\Omega} \ell(x) \cdot R u(x) \x 
 	&\le \liminf_{\eps \to 0} \Big( \F_{\eps}(u_{\eps}) - \int_{\Omega} \tilde{\ell}_{\eps}(x) \cdot R_{\eps} u_{\eps}(x) \x \Big) \\ 
 	&= \liminf_{\eps \to 0} \En_{\eps}(y_{\eps}) \\ 
  	&\le \limsup_{\eps \to 0} \En_{\eps}(y_{\eps})  \\ 
	& = \limsup_{\eps \to 0} \inf_{w \in L^p(\Omega;\R^n)} \En_{\eps}(w)  \\ 
 	&\le \limsup_{\eps \to 0} \En_{\eps}(z_{\eps}) \\ 
 	&= \lim_{\eps \to 0} \Big( \F_{\eps}(v_{\eps}) - \int_{\Omega} \tilde{\ell}_{\eps}(x) \cdot R' v_{\eps}(x) \x \Big) \\ 
 	&= \F_{\rm rel}(v) - \int_{\Omega} \ell(x) \cdot R v(x),    
\end{align*}
and so ${\cal G}^{(\ell)}_{ \rm rel }(u, R) \le {\cal G}^{(\ell)}_{ {\rm rel} }(v, R')$ and $\F^{(\ell)}_{ \rm rel }(u) \le \F^{(\ell)}_{ {\rm rel} }(v)$ in the situation of (a), respectively, (b). The choice $(v, R') = (u,R)$ furthermore gives $\lim_{\eps \to 0} \En_{\eps}(y_{\eps}) = \min {\cal G}^{(\ell)}_{ {\rm rel} }$, respectively, $\lim_{\eps \to 0} \En_{\eps}(y_{\eps}) = \min \F^{(\ell)}_{ {\rm rel} }$. This concludes the proof. 
\end{proof}

%%%%%%%%%%%%%%%%%%%%%%%%%%%%%%%%%%%%%%%%%%%%%%%%%%%%%%%%%%%%%%%%%%%%%%%%%%%%%%%%%%%%%%%%%%%%%%%%%%%%%%%%%%%%%%%%%%%%%%%%%
%%%%%%%%%%%%%%%%%%%%%%%%%%%%%%%%%%%%%%%%%%%%%%%%%%%%%%%%%%%%%%%%%%%%%%%%%%%%%%%%%%%%%%%%%%%%%%%%%%%%%%%%%%%%%%%%%%%%%%%%%
\section{Applications}
%%%%%%%%%%%%%%%%%%%%%%%%%%%%%%%%%%%%%%%%%%%%%%%%%%%%%%%%%%%%%%%%%%%%%%%%%%%%%%%%%%%%%%%%%%%%%%%%%%%%%%%%%%%%%%%%%%%%%%%%%
%%%%%%%%%%%%%%%%%%%%%%%%%%%%%%%%%%%%%%%%%%%%%%%%%%%%%%%%%%%%%%%%%%%%%%%%%%%%%%%%%%%%%%%%%%%%%%%%%%%%%%%%%%%%%%%%%%%%%%%%%

In this section we discuss three applications: single well materials, incompressible variants of martensite and nematic elastomers. We always assume that $p = 2$, $ \Omega \subset \R^{n} $ be a bounded Lipschitz domain, $ \partial \Omega_{*} \subset \partial \Omega $ be a closed subset such that $ \R^{n} \setminus \partial \Omega_{*} $ satisfies the cone condition and the loads $ \ell_{\eps} = \eps \tilde{\ell}_{\eps} $ satisfy $\tilde{\ell}_{\eps} \to \ell$ in $L^{2}(\Omega; \R^n)$.

%%%%%%%%%%%%%%%%%%%%%%%%%%%%%%%%%%%%%%%%%%%%%%%%%%%%%%%%%%%%%%%%%%%%%%%%%%%%%%%%%%%%%%%%%%%%%%%%%%%%%%%%%%%%%%%%%%%%%%%%%
\subsection{Linear elasticity for incompressible materials}
%%%%%%%%%%%%%%%%%%%%%%%%%%%%%%%%%%%%%%%%%%%%%%%%%%%%%%%%%%%%%%%%%%%%%%%%%%%%%%%%%%%%%%%%%%%%%%%%%%%%%%%%%%%%%%%%%%%%%%%%%

As a first straightforward application we extend the passage from nonlinear elasticity to linear elasticity in \cite{DalMasoNegriPercivale} to incompressible materials, cf.\ also \cite{Schmidt:08}. We assume the following for the stored energy function $ W : \Omega \times \M{n}{n} \to \R \cup \{+\i\}$ 
\begin{itemize}
\item 
$ W $ is a frame-indifferent Borel function, 
\item 
$ W(x,X) = + \i $ if $ X \not\in \SL(n) $, 
\item 
$W(x, R) = 0$ for a.e.~$ x \in \Omega $ and all $ R \in \SO(n) $, 
\item  
$ W( x , X ) \ge \alpha \, {\dist}^{2}( X , \SO(n) ) $ for a.e.~$ x \in \Omega $ and all $ X \in \SL(n) $ for some $\alpha > 0$.  
\end{itemize}

Furthermore, we assume that $W$ is twice differentiable with respect to $X$ on $\Omega \times \{ X \in \SL(n) : \dist(X, \SO(n)) < \eps \} $ for some $\eps > 0$. Then also the mapping $ \Omega \times \M{n}{n}_{\dev} \to \R$, $Z \mapsto W(x, e^Z)$ is twice differentiable with respect to $Z$ in a neighbourhood of $0$, and we denote its Hessian by 
\[ Q(x, Z) 
   = \frac{d^2}{dt^2} \bigg|_{t = 0} W(x, e^{tZ}). \] 
We assume that the Taylor expansion up to second order is uniform in $x$: 
\[ \lim_{t \to 0} t^{-2} \esssup_{x \in \Omega} \sup_{|Z| < t} |W(x, e^Z) - \tfrac{1}{2} Q(x, Z)| 
   = 0. \]

Note that, without loss of generality, we may assume that $W$ coincides on $ \Omega \times \{ X \in \SL(n) : \dist(X, \SO(n)) < \eps \} $ with a non-negative finite function $\bar{W}  : \Omega \times \{ X \in \M{n}{n} : \dist(X, \SO(n)) < \eps \} \to \R $ that is twice differentiable with respect to $X$. E.g., set $\bar{W}(x, X) = W(x, P_{\SL(n)} X)$, where $P_{\SL(n)} : \M{n}{n} \to \SL(n)$ is the projection onto $\SL(n)$ which is well-defined and smooth close to $\SL(n)$.) Since $\bar{W}(x, \cdot)$ is minimal at $\Id$ with value $0$, we see that (independently of the particular extension) 
\[ Q(x, X) 
   = D^2_X \bar{W} (x, \Id)[X, X]. \]
Our assumption can then equivalently be reformulated as 
\[ \lim_{t \to 0} t^{-2} \esssup_{x \in \Omega} \sup_{|X| < t} |\bar{W}(x, \Id + X) - \tfrac{1}{2} Q(x, X)| 
   = 0. \]

\begin{theo}\label{theo:singlewell}
Under the above assumptions, 
the functions $W_{\eps} := W$ (independent of $\eps$) and $ V := \frac{1}{2} Q $ fulfil the conditions (NL1)--(NL4), (L1)--(L4) and (C). Hence, the assertions of Theorem \ref{theo:gamma}, Theorem \ref{theo:compactness} and Corollary \ref{cor:low-energy-seq} hold true with $\overline{V} = \frac{1}{2} Q$. 
\end{theo} 

\begin{proof}
This is a direct consequence of our assumptions stated above and the fact that $Q^{\iqc} = Q$ due to $Q$ being convex. 
\end{proof}

%%%%%%%%%%%%%%%%%%%%%%%%%%%%%%%%%%%%%%%%%%%%%%%%%%%%%%%%%%%%%%%%%%%%%%%%%%%%%%%%%%%%%%%%%%%%%%%%%%%%%%%%%%%%%%%%%%%%%%%%%
\subsection{Incompressible martensitic variants}
%%%%%%%%%%%%%%%%%%%%%%%%%%%%%%%%%%%%%%%%%%%%%%%%%%%%%%%%%%%%%%%%%%%%%%%%%%%%%%%%%%%%%%%%%%%%%%%%%%%%%%%%%%%%%%%%%%%%%%%%%

It is also straightforward to extend the geometric linearization of multiple wells from \cite{Schmidt:08} to the incompressible setting. To this end, we consider a finite number $N$ of single well energies $W_{i, \eps} : \SL(n) \to \R$ satisfying
\[ W_{i, \eps}(X) \ge \alpha \, {\dist}^{2}( X , \SO(n) ) - \beta \eps^2 \] 
which are minimized at $U_i(\eps) = e^{\eps U_i} \in \SL(n)$, where $U_i \in \M{n}{n}_{\ils}$, and given by   
\begin{align*}
  W_{i, \eps}(X) 
  &= \frac{1}{2} \big\langle a_i \big( \sqrt{X^T X} - U_i(\eps) \big), 
    \sqrt{X^T X} - U_i(\eps) \big\rangle + \eps^2 w_i \\
  &\qquad + o \big( | \sqrt{X^T X} -  U_i(\eps) |^2 \big) 
\end{align*}
for $i = 1, \ldots, N$ with positive definite $a_1, \ldots, a_N \in \R^{n \times n}_{\sm}$ and minimal values $w_1, \ldots, w_N$. Then define $W_{\eps}$ by the `well minimum formula' 
\[ W_{\eps}(X) 
= 
\begin{cases}
\min_{1\le i\le N} W_{i, \eps}(X), & X \in \SL(n), \\
+ \i, & X \not\in \SL(n).
\end{cases} \]

\begin{theo}\label{theo:multiplewell}
Under the above assumptions, the functions $W_{\eps}$ and 
$V : \M{n}{n}_{\ils}  \to \R$ given by 
\[ V(Z) 
   = \frac{1}{2} \min_{1 \le i \le N} \big\langle a_i \big( X - U_i \big), 
    X - U_i \big\rangle + w_i \] 
 fulfil the conditions (NL1)--(NL4), (L1)--(L4) and (C). Hence, the assertions of Theorem \ref{theo:gamma}, Theorem \ref{theo:compactness} and Corollary \ref{cor:low-energy-seq} hold true. 
\end{theo} 

\begin{proof}
This is a direct consequence of our assumptions.  
\end{proof}

%%%%%%%%%%%%%%%%%%%%%%%%%%%%%%%%%%%%%%%%%%%%%%%%%%%%%%%%%%%%%%%%%%%%%%%%%%%%%%%%%%%%%%%%%%%%%%%%%%%%%%%%%%%%%%%%%%%%%%%%%
\subsection{Incompressible nematic elastomers}
%%%%%%%%%%%%%%%%%%%%%%%%%%%%%%%%%%%%%%%%%%%%%%%%%%%%%%%%%%%%%%%%%%%%%%%%%%%%%%%%%%%%%%%%%%%%%%%%%%%%%%%%%%%%%%%%%%%%%%%%%

We consider a nonlinear model for nematic elastomers proposed by Bladon, Terentjev and Warner, see \cite{BladonTerentjevWarner:93}, and analyzed by DeSimone and Dolzmann in \cite{DeSimoneDolzmann}. Here the stored energy is given by 
\[ W_{\eps}(X) = \frac{ \sigma_{1}(X)^{2} }{ \gamma_{\eps,1}^2 } 
+ \frac{ \sigma_{2}(X)^2 }{ \gamma_{\eps,2}^2 } 
+ \frac{ \sigma_{3}(X)^2 }{ \gamma_{\eps,3}^2 } - 3 \]
for $\det X = 1$ and is $\infty$ otherwise. The $ \sigma_{i}(X) $ are the singular values of $X$ in their ascending order and $0 < \gamma_{\eps,1} \le \gamma_{\eps,2} \le \gamma_{\eps,3}$ are $\eps$-dependent parameters such that $\gamma_{\eps,1} \gamma_{\eps,2} \gamma_{\eps,3} = 1$, which describe the preferred strains. Indeed, one has 
\[ \operatorname{arg\,min} W_{\eps}
   = \{ X \in \SL(3) : W_{\eps}(X) = 0 \} 
   = \bigcup_{\{e_1,e_2,e_3\} \in \mathcal{R}} \SO(3) \sum_{i=1}^3 \gamma_{\eps,i} e_i \otimes e_i, 
\]
where $\mathcal{R}$ is the set of orthonormal bases of $\R^3$, cf.\ \cite{DeSimoneDolzmann}. We make the assumption that the preferred strains are $O(\eps)$ close to $\SO(3)$, i.e., the $\gamma_{\eps,i}$ are of the form $\gamma_{\eps,i} = 1 + \eps \rho_i + o(\eps)$ for some $\rho_i$, $i = 1,2,3$ or, equivalently,  
\[ \gamma_{\eps,i} 
   = e^{\eps \rho_i(\eps)} 
   = e^{\eps \rho_i + o(\eps)}, \quad i = 1,2,3, \] 
for suitable $\rho_i(\eps) \to \rho_i$ as $\eps \to 0$. Notice that $\sum_{i=1}^3 \rho_{i}(\eps) = 0$ for all $\eps$. 
\bigskip

\begin{theo}\label{theo:nematicelastomers}
Under the above assumptions, the functions $W_{\eps}$ and 
$V : \M{3}{3}  \to \R$ given by 
\[ V(Z) 
   =
\begin{cases}
2 \sum_{i=1}^3 \big( \lambda_{i}(Z_{\sm}) - \rho_i \big)^2, & \tr Z = 0 , \\
+ \i, & \tr Z \ne 0
\end{cases} \] 
fulfil the conditions (NL1)--(NL4), (L1)--(L4) and (C). Hence, the assertions of Theorem \ref{theo:gamma}, Theorem \ref{theo:compactness} and Corollary \ref{cor:low-energy-seq} hold true. 
\end{theo}

In fact, as in this particular case the quasiconvex envelope of $ W_{\eps} $ is known, see Theorem~4 in \cite{DeSimoneDolzmann}, besides providing a general link between the nonlinear and linear theory, Theorem~\ref{theo:nematicelastomers} can also be used to identify the density of the limit functional. Generalizing \cite{CesanaDeSimone:11,Cesana}, here we also provide an explicit expression if not necessarily two of the constants $\rho_{1}, \rho_{2}, \rho_{3}$ are equal. 

\begin{prop}\label{prop:special-iqc-envelope}
Let $ \rho_{1} \le \rho_{2} \le \rho_{3} $ fulfil $ \rho_{1} + \rho_{2} + \rho_{3} = 0 $.
The iqc-envelope of the function
\[ f : \M{3}{3}_{\dev}  \to \R, \quad f(Z) = 2 \sum_{i=1}^3 \big( \lambda_{i}(Z_{\sm}) - \rho_i \big)^2
\] 
where $ \lambda_{i} = \lambda_{i}(Z_{\sm}) $ denote the eigenvalues of $ Z_{\sm} $ in the ascending order
is given by
\[
f^{\iqc}( Z )
= \left\{ \begin{array}{ll}
0, & \lambda_{1} \ge \rho_{1} ,\ \lambda_{3} \le \rho_{3}, \\
3 ( \lambda_{1} - \rho_{1} )^{2} , &
\lambda_{1} \le \rho_{1}, \ \lambda_{3} - \rho_{3} \le \lambda_{2} - \rho_{2}, \\
f(Z)
, &
\lambda_{1} - \rho_{1} \le \lambda_{2} - \rho_{2} \le \lambda_{3} - \rho_{3}, \\
3 ( \lambda_{3} - \rho_{3} )^{2}, &
\lambda_{2} - \rho_{2} \le \lambda_{1} - \rho_{1}, \  \lambda_{3} \ge \rho_{3}.
\end{array} \right.  \]
\end{prop} 

\begin{proof}[Proof of Theorem~\ref{theo:nematicelastomers}]
In order to check that $W_{\eps}$ satisfies the required lower bound, we first note that an elementary analysis shows that on $\{ \mu = (\mu_1, \mu_2, \mu_3) \in (0,\infty)^3 : \mu_1 \mu_2 \mu_3 = 1 \}$ the mapping $\mu = (\mu_1, \mu_2, \mu_3) \mapsto \mu_1^2 + \mu_2^2 + \mu_3^2$ is minimized precisely at $\mu = (1,1,1)$ and satsifies $\mu_1^2 + \mu_2^2 + \mu_3^2 \ge 3 + c \sum_{i=1}^3 (\mu_i-1)^2$ for some $c > 0$ there. Using this for $\mu_i = \frac{ \sigma_{i}(X) }{ \gamma_{\eps,i} }$, $i = 1,2,3$, combined with the quadratic growth of $W_{\eps}$ at $\infty$, shows that 
\begin{align*}  
  W_{\eps}(X) 
  &\ge c \sum_{i=1}^3 \Big( \frac{ \sigma_{i}(X) }{ \gamma_{\eps,i} } - 1 \Big)^2 
   \ge \frac{c}{2} \sum_{i=1}^3 ( \sigma_{i}(X) - \gamma_{\eps,i} \big)^2 \\ 
  &\ge \frac{c}{4} \sum_{i=1}^3 ( \sigma_{i}(X) - 1 \big)^2 - C \eps^2 
   = \frac{c}{4} \dist^2(X, \SO(3)) - C \eps^2.  
\end{align*}

We have to compute $V_{\eps} : \M{3}{3}_{\dev} \to \R$, $V_{\eps}(Z) = \frac{1}{\eps^2} W_{\eps}(e^{\eps Z})$. Denoting by $\lambda_1(X) \le \lambda_2(X) \le \lambda_3(X)$ the eigenvalues of a symmetric matrix $X \in \M{3}{3}_{\sm}$, by Taylor expansion we have 
\[ \sigma_i ( e^{\eps Z} ) 
= \lambda_i \big( \sqrt{ e^{\eps Z^{T} } e^{\eps Z}  } \big) 
= 1 + \eps \lambda_{i}( Z_{\sm} ) + O( \eps^{2} ) \]
for $Z \in \M{3}{3}_{\dev}$, $i = 1,2,3$, since 
\begin{align*}  
\sqrt{ e^{ \eps Z^{T} } e^{\eps Z}  } 
& = \sqrt{ ( \Id + \eps Z^{T} + O( \eps^{2} ) )( \Id + \eps Z + O( \eps^{2} ) )  } \\
& = \sqrt{ \Id + 2 \eps Z_{\sm} + O( \eps^{2} ) } \\
& = \Id + \eps Z_{\sm} + O( \eps^{2} ).
\end{align*}
Furthermore, 
\[ \sum_{i=1}^{3} \sigma_i ( e^{\eps Z} )^{2}
= \sum_{i=1}^{3} \lambda_i \big( \sqrt{ e^{\eps Z^{T} } e^{\eps Z}  } \big)^{2}
= | \sqrt{ e^{\eps Z^{T} } e^{\eps Z} } |^{2} 
= | e^{\eps Z} |^{2}. \] 
From $ \tr Z = 0 $, it follows
\begin{align*} 
| e^{\eps Z} |^{2}
& = \left( \Id + \eps Z + \frac{\eps^{2}}{2} Z^{2} + O( \eps^{3} ) \right) : 
	\left( \Id + \eps Z + \frac{\eps^{2}}{2} Z^{2} + O( \eps^{3} ) \right) \\
& = 3 + \eps^{2} Z : Z + \eps^{2} Z^{2} : \Id + O( \eps^{3} ) \\
& = 3 + 2 \eps^{2} | Z_{\sm} |^{2} + O( \eps^{3} ).
\end{align*}
The error terms are uniform on bounded subsets of $\M{3}{3}_{\dev}$. Expanding  
\[ \gamma_{\eps,i}^{-2} 
   = e^{- 2 \eps \rho_i(\eps)} 
   = 1 - 2 \eps \rho_i(\eps) + 2 \eps^2 \rho_i^2(\eps) + O(\eps^3) \] 
and using that $\sum_{i=1}^3 \rho_i(\eps) = 0$, we find that  
\begin{align*}
  W_{\eps}( e^{ \eps Z } ) 
  &= - 3 + \sum_{i=1}^3 \sigma_{i}(e^{\eps Z})^2 \big( 1 - 2 \eps \rho_i(\eps) + 2 \eps^2 \rho_i^2(\eps) + O(\eps^3) \big) \\ 
  &= - 3 + | e^{\eps Z} |^{2} + \sum_{i=1}^3 \big( 1 + 2 \eps \lambda_{i}(Z_{\sm}) + O( \eps^{2} ) \big) \big( - 2 \eps \rho_i(\eps) + 2 \eps^2 \rho_i^2(\eps) + O( \eps^{3} ) \big) \\ 
  &= 2 \eps^2 |Z_{\sm}|^2 + \sum_{i=1}^3 \big( 2 \eps^2 \rho_i^2(\eps) - 4 \eps^2 \rho_i(\eps) \lambda_{i}(Z_{\sm}) \big)  + O( \eps^{3} ) 
\end{align*}
Hence,
\[ V_{ \eps }(Z) = 2 |Z_{\sm}|^2 + \sum_{i=1}^3 \big( 2 \rho_i^2(\eps) - 4 \rho_i(\eps) \lambda_{i}(Z_{\sm}) \big)  + O( \eps ) \]
converges uniformly in $Z$ on compact subsets of $\M{3}{3}_{\dev}$ to 
\[ V(Z) := 2 \sum_{i=1}^3 \big( \lambda_{i}(Z_{\sm}) - \rho_i \big)^2. \]
This function satisfies the required growth assumptions. The necessary Lipschitz estimate follows from 
\[ V(Z) - V(Z') 
   = 2 \sum_{i=1}^3 \big( - 2 \rho_i + \lambda_{i}(Z_{\sm}) + \lambda_{i}(Z_{\sm}') \big) \big( \lambda_{i}(Z_{\sm}) - \lambda_{i}(Z_{\sm}') \big). \qedhere \]
\end{proof}

The quasiconvex envelope of $ W_{\eps} $ is known, see Theorem 4 in \cite{DeSimoneDolzmann}. 
For $ X \in \M{n}{n} $ with $ \det X = 1 $ is given by
\[ W_{\eps}^{\qc}(X)
= \left\{ \begin{array}{ll}
0, & \frac{ \sigma_{1}(X) }{ \gamma_{\eps,1} } \ge 1, \ \frac{ \sigma_{3}(X) }{ \gamma_{\eps,3} } \le 1, \\
( \frac{ \sigma_{1}(X) }{ \gamma_{\eps,1} } )^{2} + 2 \frac{ \gamma_{\eps,1} }{ \sigma_{1}(X) } - 3, &
\frac{ \sigma_{1}(X) }{ \gamma_{\eps,1} } \le 1, \ \frac{ \sigma_{3}(X) }{ \gamma_{\eps,3} } \le \frac{ \sigma_{2}(X) }{ \gamma_{\eps,2} }, \  \\
W_{\eps}(X)
%( \frac{ \sigma_{1}(X) }{ \gamma_{\eps,1} } )^{2} + ( \frac{ \sigma_{2}(X) }{ \gamma_{\eps,2} } )^{2} + ( \frac{ \sigma_{3}(X) }{ \gamma_{\eps,3} } )^{2} - 3 
, &
\frac{ \sigma_{1}(X) }{ \gamma_{\eps,1} } \le \frac{ \sigma_{2}(X) }{ \gamma_{\eps,2} } \le \frac{ \sigma_{3}(X) }{ \gamma_{\eps,3} }, \\
( \frac{ \sigma_{3}(X) }{ \gamma_{\eps,3} } )^{2} + 2 \frac{ \gamma_{\eps,3} }{ \sigma_{3}(X) } - 3, &
\frac{ \sigma_{2}(X) }{ \gamma_{\eps,2} } \le \frac{ \sigma_{1}(X) }{ \gamma_{\eps,1} }, \ \frac{ \sigma_{3}(X) }{ \gamma_{\eps,3} } \ge 1.
\end{array} \right.  \]
and equals $ + \i $ if $ \det X \ne 1 $. It may be viewed upon as
\[ W_{\eps}^{\qc}(X) = \inf_{ u \in w + W^{1,p}_{0}( \Omega ) } \int_{\Omega} W_{\eps}( \nabla u(x) ) \x  \]
for $ \Omega := (0,1)^{3} $ and $ w(x) := X x $. 

Let us take arbitrary $ Z \in \M{n}{n}_{\dev} $. Define $ w_{\eps}(x) := e^{\eps Z} x $ (and correspondingly $ g_{\eps}(x) = \frac{ e^{\eps Z} - \Id }{ \eps } x $), $ \partial \Omega_{*} := \partial \Omega $ and $ \F_{\eps} $ as in Theorem~\ref{theo:gamma}. Since $ g_{\eps} \to g $ in $ W^{1,\i}( \Omega ; \R^{n} ) $ with $ g(x) = Z x $, it follows from the properties of $ \Gamma $-convergence and the representation formula from Theorem~\ref{theo:iqc} that
\[ \frac{1}{\eps^2} W_{\eps}^{\qc}(e^{\eps Z}) 
= \inf \F_{\eps} 
\to \min \F_{\rm rel}
= \overline{V} ( Z_{\sm} ). \] 
Since $f$ from Proposition~\ref{prop:special-iqc-envelope} equals $ V|_{ \M{n}{n}_{\dev} } $, we can determine $f^{\iqc}$ by using this convergence. 

\begin{proof}[Proof of Proposition~\ref{prop:special-iqc-envelope}]
Since $ V^{\iqc} $ is completely determined on $ \M{n}{n}_{\ils} $, let us take arbitrary $ Z \in \M{n}{n}_{\ils} $.
Then 
\[ \sigma_{i}( e^{ \eps Z } ) = \lambda_{i}( e^{ \eps Z } ) = e^{ \eps  \lambda_{i}(Z) } \]
and since 
\[ \frac{ \sigma_{i}( e^{ \eps Z } ) }{ \gamma_{\eps,i} } 
= e^{ \eps \lambda_{i}(Z) - \eps \rho_{i} - o(\eps) } 
= 1 + \eps ( \lambda_{i}(Z) - \rho_{i} ) + o(\eps) \]

Therefore, for small $ \eps $ the case $ \frac{ \sigma_{i}( e^{\eps Z} ) }{ \gamma_{\eps,i} } > 1 $ is equivalent to $ \lambda_{i}(Z) - \rho_{i} > 0 $.
By employing for 
\[ x = e^{ \eps \lambda_{i}(Z) - \eps \rho_{i}( \eps ) } - 1
= \eps ( \lambda_{i}(Z) - \rho_{i} ) + o( \eps ), \]
the approximation
\[ (1+x)^{2} + \frac{2}{1+x} - 3 
= 1 + 2x + x^2 + 2 ( 1 - x + x^{2} + o( x^{2} ) ) - 3
= 3 x^{2} + o( x^{2} ), \]
we arrive at
\[ \frac{ ( \frac{ \sigma_{i}( e^{ \eps Z } ) }{ \gamma_{\eps,i} } )^{2} + 2 \frac{ \gamma_{\eps,i} }{ \sigma_{i}( e^{ \eps Z } ) } - 3 }{ \eps^{2} }
= \frac{ 3 ( \eps ( \lambda_{i}(Z) - \rho_{i} ) + o( \eps ) )^{2} + o( \eps^{2} ) }{ \eps^{2} } 
= 3 ( \lambda_{i}(Z) - \rho_{i} )^{2} + o( 1 ). \]
Hence, with $ \lambda_{i} = \lambda_{i}(Z) $
\[
f^{\iqc}( Z )
= \left\{ \begin{array}{ll}
0, & \sigma(Z) \subset [ \rho_{1} , \rho_{3} ], \\
3 ( \lambda_{1} - \rho_{1} )^{2} , &
\lambda_{1} \le \rho_{1}, \ \lambda_{3} - \rho_{3} \le \lambda_{2} - \rho_{2}, \\
V(Z)
%( \lambda_{1} - \rho_{1} )^{2} + ( \lambda_{2} - \rho_{2} )^{2} + ( \lambda_{3} - \rho_{3} )^{2}
, &
\lambda_{1} - \rho_{1} \le \lambda_{2} - \rho_{2} \le \lambda_{3} - \rho_{3}, \\
3 ( \lambda_{3} - \rho_{3} )^{2}, &
\lambda_{2} - \rho_{2} \le \lambda_{1} - \rho_{1}, \  \lambda_{3} \ge \rho_{3}.
\end{array} \right.  \]
\end{proof}
By employing $ \sum_{i=1}^{3} ( \lambda_{i} - \rho_{i} ) = 0 $, we may present the formula in an alternative form that perhaps more clearly shows what happens by this kind of convexification: 
\[
\frac{1}{2} f^{\iqc}( Z )
= \left\{ \begin{array}{ll}
0, & \sigma( Z ) \subset [ \rho_{1} , \rho_{3} ], \\
( \lambda_{1} - \rho_{1} )^{2} + 2 \left( \frac{ ( \lambda_{2} - \rho_{2} ) + ( \lambda_{3} - \rho_{3} ) }{2} \right)^{2}, &
\lambda_{1} \le \rho_{1}, \ \lambda_{3} - \rho_{3} \le \lambda_{2} - \rho_{2}, \\
( \lambda_{1} - \rho_{1} )^{2} + ( \lambda_{2} - \rho_{2} )^{2} + ( \lambda_{3} - \rho_{3} )^{2}, &
\lambda_{1} - \rho_{1} \le \lambda_{2} - \rho_{2} \le \lambda_{3} - \rho_{3}, \\
2 \left( \frac{ ( \lambda_{1} - \rho_{1} ) + ( \lambda_{2} - \rho_{2} ) }{2} \right)^{2} + ( \lambda_{3} - \rho_{3} )^{2}, &
\lambda_{2} - \rho_{2} \le \lambda_{1} - \rho_{1}, \  \lambda_{3} \ge \rho_{3}.
\end{array} \right.  \]
%

%------------------------------------------------------------------------------------------------------------------------
%%%%%%%%%%%%%%%%%%%%%%%%%%%%%%%%%%%%%%%%%%%%%%%%%%%%%%%%%%%%%%%%%%%%%%%%%%%%%%%%%%%%%%%%%%%%%%%%%%%%%%%%%%%%%%%%%%%%%%%%%
%%%%%%%%%%%%%%%%%%%%%%%%%%%%%%%%%%%%%%%%%%%%%%%%%%%%%%%%%%%%%%%%%%%%%%%%%%%%%%%%%%%%%%%%%%%%%%%%%%%%%%%%%%%%%%%%%%%%%%%%%
\appendix
\section{Appendix}
%%%%%%%%%%%%%%%%%%%%%%%%%%%%%%%%%%%%%%%%%%%%%%%%%%%%%%%%%%%%%%%%%%%%%%%%%%%%%%%%%%%%%%%%%%%%%%%%%%%%%%%%%%%%%%%%%%%%%%%%%
%%%%%%%%%%%%%%%%%%%%%%%%%%%%%%%%%%%%%%%%%%%%%%%%%%%%%%%%%%%%%%%%%%%%%%%%%%%%%%%%%%%%%%%%%%%%%%%%%%%%%%%%%%%%%%%%%%%%%%%%%

For easy reference we list some basic results on relaxation, the Bogovskii operator and a particular form of the extension theorem for Sobolev mappings. 

\subsection{Relaxation}

Statement~III.7 from \cite{AcerbiFusco} establishes a relation between the relaxation of an integral functional and the quasiconvexification of its density:
\begin{theo}
\label{theo:lsc}
Let $ \Omega \subset \R^{n} $ be a bounded open set.
Suppose $ f: \Omega \times \M{m}{n} \to \R $ is a Carath\'{e}odory function which
satisfies for some $ p \ge 1 $ and $ \beta > 0 $
\[ 0 \le f(x,X) \le \beta( 1 + |X|^{p} ) \]
for almost all $ x \in \Omega $ and all $ X \in \M{m}{n} $.
Define for $ u \in L^{p}( \Omega ; \R^{m} ) $
\[ \F(u) := 
\left\{
\begin{array}{cl}
\int_{ \Omega } f( x , \nabla u(x) ) \x, & u \in W^{1,p}( \Omega ; \R^{m} ), \\
\i, & \mbox{else.}
\end{array} \right. \]
The sequentially weakly lower semicontinuous envelope of $ \F|_{ W^{1,p}( \Omega ; \R^{m} ) } $ is given by
\[ \swlsc \F|_{ W^{1,p}( \Omega ; \R^{m} ) }( u ) = \int_{ \Omega } f^{\qc}( x , \nabla u(x) ) \x \]
for each $ u \in W^{1,p}( \Omega ; \R^{m} ) $.
\end{theo}
%------------------------------------------------------------------------------------------------------------------------
The following complementary result that is a reduced version of Theorem 9.8 in \cite{Dacorogna} is useful in cases with given boundary data:
%------------------------------------------------------------------------------------------------------------------------
\begin{lemma}
\label{lemma:Dacorogna}
Let $ \Omega \subset \R^{n} $ be a bounded open set and $ 1 \le p < \i $.
Let $ f : \Omega \times \M{m}{n} \to \R $ be a Carath\'{e}odory function satisfying
$ 0 \le | f(x,X) | \le \beta ( 1 + |X|^{p} ) $ for every $ X \in \M{m}{n} $.
If $ u \in W^{1,p}( \Omega ; \R^{m} ) $, 
then there exist $ u_{j} \in u + W^{1,p}_{0}( \Omega ; \R^{m} ) $ such that
\[ \| u_{j} - u \|_{ L^{p}( \Omega ) } \to 0 
\quad \mbox{and} \quad
\lim_{ j \to \i } \int_{ \Omega } f( x , \nabla u_{j}(x) ) \x = \int_{ \Omega } f^{\qc}( x , \nabla u(x) ) \x. \]
\end{lemma} 
%------------------------------------------------------------------------------------------------------------------------

%%%%%%%%%%%%%%%%%%%%%%%%%%%%%%%%%%%%%%%%%%%%%%%%%%%%%%%%%%%%%%%%%%%%%%%%%%%%%%%%%%%%%%%%%%%%%%%%%%%%%%%%%%%%%%%%%%%%%%%%%
\subsection{Bogovskii operator}

The so called Bogovskii operator ensures the solvability of the problem $ \DIV v = f $ 
within $ W^{1,p}_{0}( \Omega ; \R^{n} ) $ for functions $f$ with zero mean. 
The existence of such an operator was shown in, e.g., \cite{Bogovskii,Galdi,BorchersSohr}.
\begin{theo}
\label{theo:Bogovskii}
Let $ \Omega \subset \R^{n} $ be a bounded domain having the cone property and $ 1 < p < \i $.
There exists a linear operator $ {\mathcal B} = {\mathcal B}_{\Omega,p} : L^{p}( \Omega ) \to W^{1,p}_{0}( \Omega ; \R^{n} ) $ 
with the following properties:
\begin{itemize}
\item 
For every $ f \in L^{p}( \Omega ) $ with $ \int_{\Omega} f(x) \x = 0 $, it holds
\[ \DIV {\mathcal B} f = f. \]
\item 
$ {\mathcal B} $ is bounded. Namely, for every $ f \in L^{p}( \Omega ) $
\[ \| \nabla ( {\mathcal B} f ) \|_{ L^{p}( \Omega ; \M{n}{n} ) } \le C \| f \|_{ L^{p}( \Omega ) }. \]
The constant $C$ depends only on $ \Omega $ and $p$, and is translation- and scaling-invariant.
\item 
If $ f \in C^{\i}_{c}( \Omega ) $, then $ {\mathcal B} f \in C^{\i}_{c}( \Omega ; \R^{n} ) $.
\end{itemize}
\end{theo}
With this result one may show the following density result with constraint on the divergence, see Theorem III.4.1 in \cite{Galdi}:
\begin{theo}
\label{theo:density-div-free}
If $ U \subset \R^{n} $ is a bounded domain satisfying the cone property, then
$ \{ u \in C_{c}^{\i}( U ) : \DIV u = 0 \} $ 
is dense in 
$ \{ u \in W^{1,p}_{ 0 }( U, \R^{n} ) : \DIV u = 0 \} $.
\end{theo}
%

%%%%%%%%%%%%%%%%%%%%%%%%%%%%%%%%%%%%%%%%%%%%%%%%%%%%%%%%%%%%%%%%%%%%%%%%%%%%%%%%%%%%%%%%%%%%%%%%%%%%%%%%%%%%%%%%%%%%%%%%%
\subsection{Extension theorem}

Let $ \Omega \subset \R^{n} $ be an open set and $ D \subset \overline{ \Omega } $ a closed subset.
The space $ W^{1,p}_{ D }( \Omega ; \R^{n} ) $ is defined as the closure of $\{ u|_{ \Omega } : u \in C_{c}^{\i}( \R^{n} \setminus D ; \R^{n} ) \} $ in $ W^{1,p}( \Omega ; \R^{n} ) $.
Its properties were thoroughly explored in \cite{BMMM}. 
We state a simplified version of Theorem 1.3 from that article regarding existence of an extension operator:
\begin{theo}
\label{theo:extension}
Suppose that $ \Omega \subset \R^{n} $ is a bounded Lipschitz domain 
and $ D \subset \overline{ \Omega } $ be a closed subset.
For any $ p \in [1,\i] $ there exists a bounded linear operator
\[ E : W^{1,p}_{D}( \Omega ) \to W^{1,p}_{D}( \R^{n} ) \]
such that $ ( E u )|_{ \Omega } = u $ a.e.~on $ \Omega $.
\end{theo}

%%%%%%%%%%%%%%%%%%%%%%%%%%%%%%%%%%%%%%%%%%%%%%%%%%%%%%%%%%%%%%%%%%%%%%%%%%%%%%%%%%%%%%%%%%%%%%%%%%%%%%%%%%%%%%%%%%%%%%%%%
%%%%%%%%%%%%%%%%%%%%%%%%%%%%%%%%%%%%%%%%%%%%%%%%%%%%%%%%%%%%%%%%%%%%%%%%%%%%%%%%%%%%%%%%%%%%%%%%%%%%%%%%%%%%%%%%%%%%%%%%%
\section*{Acknowledgment}
%%%%%%%%%%%%%%%%%%%%%%%%%%%%%%%%%%%%%%%%%%%%%%%%%%%%%%%%%%%%%%%%%%%%%%%%%%%%%%%%%%%%%%%%%%%%%%%%%%%%%%%%%%%%%%%%%%%%%%%%%
%%%%%%%%%%%%%%%%%%%%%%%%%%%%%%%%%%%%%%%%%%%%%%%%%%%%%%%%%%%%%%%%%%%%%%%%%%%%%%%%%%%%%%%%%%%%%%%%%%%%%%%%%%%%%%%%%%%%%%%%%

Both authors thank the hospitality of the Mathematisches Forschungsinstitut Oberwolfach, where part of this work was carried out.

%--------------------------------------------------------------------------

%%%%%%%%%%%%%%%%%%%%%%%%%%%%%%%%%%%%%%%%%%%%%%%%%%%%%%%%%%%%%%%%%%%%%%%%%%%%%%%%%%%%%%%%%%%%%%%%%%%%%%%%%%%%%%%%%%%%%%%%%
%%%%%%%%%%%%%%%%%%%%%%%%%%%%%%%%%%%%%%%%%%%%%%%%%%%%%%%%%%%%%%%%%%%%%%%%%%%%%%%%%%%%%%%%%%%%%%%%%%%%%%%%%%%%%%%%%%%%%%%%%
 \typeout{References}

\end{document}